\documentclass[12pt,
paper=a4,
headings=small,
]{scrartcl}

\usepackage[a4paper,left=4cm, right=4cm]{geometry}    
\setkomafont{sectioning}{\rmfamily\bfseries} 

\usepackage{pifont}
\usepackage{times}
\usepackage{amsfonts}       %
\usepackage{amsmath} 
\usepackage{amssymb}        %
\usepackage{amsthm}         
\usepackage{booktabs}

\usepackage{latexsym}
 \usepackage[matrix,curve]{xypic}
\usepackage{microtype}        

\usepackage{rotating}

\newcommand{\N}{\mathbb{N}}
\newcommand{\Z}{\mathbb{Z}}

\newcommand{\C}{\mathbb{C}}

\newtheorem{thm}{Theorem}[section]
\newtheorem{cor}[thm]{Corollary}
\newtheorem{lem}[thm]{Lemma}
\newtheorem{prop}[thm]{Proposition}

\theoremstyle{definition}

\newcommand\nosf[1]{\begin{footnotesize}\textup{\textsf{#1}}\end{footnotesize}}
\usepackage{enumitem}
\setlist[enumerate]{itemsep=-0.5ex plus0.1ex minus 0.2ex}
\setlist[description]{itemsep=-0.5ex plus0.1ex minus 0.2ex}
\setlist[itemize]{itemsep=-0.5ex plus0.1ex minus 0.2ex}

\everydisplay{
\thickmuskip=5mu plus 2mu
\medmuskip=4mu plus 1mu
\relax}

\begin{document}

\centerline{\textbf{ Exceptional poles of local $L$-functions for $GSp(4)$}} 
\centerline{\textbf{with respect to split Bessel models}}

\bigskip\noindent

\centerline{ Rainer Weissauer (Heidelberg)}

\bigskip\noindent  

\bigskip\noindent
\textbf{Introduction}.
For the rank two symplectic group $G$ 
over a nonarchimedean local field $k$ of characteristic $char(k)\neq 2$ and for its similitude character $\lambda_G: G \to k^*$ we study $H$-period functionals, i.e. $\C$-linear functionals 
$$ \ell: \Pi \to (\C, \rho\circ\lambda_G) $$ on irreducible admissible 
representations $\Pi$ of $G$ that are equivariant with respect to a subgroup $H$ of $G$ that contains $Sl(2,k)\times Sl(2,k)$, the center $Z_G$ of $G$ and a certain torus $T$. Here $\rho$ is a character of $k^*$.  
If nontrivial $H$-period functionals exist at all, the corresponding character $\rho$
is uniquely determined by $\Pi$ and the dimension of the space of such functionals is $\leq 1$. If the  central character of $\Pi$ is trivial, this is related to more general results for orthogonal groups \cite{AGRS}, \cite{Wa}. We determine the irreducible admissible representations of $G$ for which such functionals exist (theorem \ref{maintheorem}).
Except for one-dimensional representations, the corresponding representations $\Pi$  turn out to be either Saito-Kurokawa type representations or representations of type \nosf{VId}, in the notation of \cite{Sally-Tadic}, and for simplicity
we refer to them all as representations of extended Saito-Kurokawa type. It is well known that
these are paramodular \cite{Roberts-Schmidt},   
have a unique split Bessel model $\Lambda$ \cite{Roberts-Schmidt_Bessel},  and the corresponding Bessel functional can be obtained from 
the functional $\ell$ above. 

\bigskip
To any Bessel model $\Lambda$ of an irreducible, admissible representation $\Pi$ of $G$, Piatetskii-Shapiro \cite{PS-L-Factor_GSp4} associated a local $L$-factor $L^{PS}(s,\Pi,\Lambda)$, defined as a product
 $$L^{PS}\!(s,\Pi,\Lambda)=L_{ex}^{PS}\!(s,\Pi,\Lambda)L_{reg}^{PS}(s,\Pi,\Lambda)$$
 of an exceptional $L$-factor $L_{ex}^{PS}(s,\Pi,\Lambda)$ and a regular $L$-factor $L_{reg}^{PS}(s,\Pi,\Lambda)$. As already observed by Piatetskii-Shapiro, a nontrivial
exceptional $L$-factor $L_{ex}^{PS}\!(s,\Pi, \Lambda)$ for a split Bessel model $\Lambda$
of $\Pi$ gives rise to a nontrivial $H$-period functional on $\Pi$. For split Bessel models we show the converse: The exceptional $L$-factor $L_{ex}^{PS}\!(s,\Pi,\Lambda)$ is nontrivial if and only if $\Pi$ admits a nontrivial $H$-equivariant functional, confirming an expectation formulated by Piatetskii-Shapiro \cite{PS-L-Factor_GSp4} for arbitrary Bessel models. 
For split Bessel models the regular part of the $L$-factor is $L_{reg}^{PS}(s,\Pi,\Lambda)= L_{sreg}^{PS}(s,\Pi,\Lambda)L(\mu\otimes M,s) $ and the two factors on the right side are described in \cite{RW2} and \cite{RW}.  In this paper we determine $L_{ex}^{PS}\!(s,\Pi,\Lambda)$ for split Bessel models and therefore complete the computation of the local $L$-factors $L^{PS}(s,\Pi,\Lambda)$ in the split cases. For anisotropic Bessel models the regular $L$-factors $L_{reg}^{PS}(s,\Pi,\Lambda)$ have been determined
by Danisman \cite{Danisman}, \cite{Danisman2}, \cite{Danisman3} for all $\Pi$.
Furthermore, the exceptional part  $L_{ex}^{PS}(s,\Pi,\Lambda)$ has been computed  
for anisotropic Bessel models $\Lambda$ and non-generic cuspidal representations in \cite{Danisman3bis}. 

\bigskip
Our main result
on the local Piatetskii-Shapiro $L$-function $L^{PS}(s,\Pi,\Lambda)$ for split
Bessel models,  in terms of the local Tate
$L$-factors, is contained in table 1 above. 
For a local Saito-Kurokawa representation $\Pi$, attached to an irreducible generic smooth representation $\pi$ of $Gl(2,k)$ with trivial central character, and $\Pi$ suitably normalized by a character twist (e.g. this replaces \nosf{Vc} by its twist \nosf{Vb}), the results of the table above can be restated in the following simple form $$ L^{PS}(s,\Pi,\Lambda)= L(s-\frac{1}{2})L(\pi,s)L(s+\frac{1}{2}) \ .$$ 
Of course, this is expected for the local Saito-Kurokawa representations $\Pi$. Since the present paper is a sequel of \cite{RW},  notation from there is used to considerable extent.

\begin{table}
\begin{tabular}{llll}
\toprule
$\Pi$&  $\ \ L^{PS}_{ex}(s,\Pi,\Lambda) $& $\ \ L^{PS}_{sreg}(s,\Pi,\Lambda)$ & $\ \ L(\mu\otimes M,s)$ \\
\midrule
IIb &  $ L(\nu^{1/2}\chi\sigma, s)$   & $1$  & $L(\nu^{-1/2}\chi\sigma, s)L(\sigma, s)  L(\chi^2\sigma, s)$ \\
Vb &    $  L(\nu^{1/2}\sigma, s)  $ & $ 1 $  & $ L(\nu^{-1/2}\sigma, s) L(\nu^{1/2}\xi\sigma, s)$ \\
Vc &    $ L(\nu^{1/2}\xi\sigma, s)  $ & $ 1 $ &  $ L(\nu^{-1/2}\xi\sigma, s) L(\nu^{1/2}\sigma, s)$ \\
VIc &    $ L(\nu^{1/2}\sigma, s) $ & $ L(\nu^{1/2}\sigma, s)$ & $ L(\nu^{-1/2}\sigma, s)$   \\
VId &    $ L(\nu^{1/2}\sigma, s) $ & $ L(\nu^{1/2}\sigma, s) $ & $ L(\nu^{-1/2}\sigma, s)^2$ \\
XIb &    $ L(\nu^{1/2}\sigma, s) $ &  $ 1 $ &  $ L(\nu^{-1/2}\sigma, s)$  \\
\bottomrule
\end{tabular}
\caption{
List of exceptional Piateskii-Shapiro $L$-factors for split Bessel models.
For representations $\Pi$   
 not isomorphic 
to one contained in the table, $L_{ex}^{PS}(s,\Pi,\Lambda)=1$
and $L^{PS}(s,\Pi,\Lambda)=L_{reg}^{PS}(s,\Pi,\Lambda)$
holds for all split Bessel models $\Lambda$. The notation for representations $\Pi$ in the left column uses the classification symbols of \cite{Sally-Tadic}
and \cite{Roberts-Schmidt}.}
\end{table}

\bigskip\noindent
\textbf{Outline of the proof}. As the first step for the computation of the exceptional poles of the Piatetskii-Shapiro $L$-function, we prove that $H$-period functionals only exist
for the class of extended local Saito-Kurokawa representations.  For this, and also the further  steps, we investigate a certain quotient space $\overline\Pi$ of the representation space of $\Pi$ endowed with an action of 
the affine linear group $Gl_a(2,k)$, where $Gl_a(2,k)$ denotes the semidirect product of the linear group $Gl(2,k)$ and its two dimensional standard representation on $k^2$.
By abuse of notation we often write $Gl_a(2), Gl(2)$ etc. instead of $Gl_a(2,k),Gl(2,k)$ etc.
if the meaning is clear from the context.

\medskip
For a t.d. (totally disconnected topological) group $G$ let ${\cal C}_G$ be the category of smooth complex representations
of $G$ and ${\cal C}_G^{fin}\subseteq {\cal C}_G$ the full subcategory of representations of finite length. In the previous paper \cite{RW} we considered $\overline\Pi\in {\cal C}_{Gl_a(2)} $ and Mellin functors $k_\rho: {\cal C}_{Gl_a(n)} \to {\cal C}_{Gl_a(n-1)}$. In particular we use the 
functors $k_\rho: {\cal C}_{Gl_a(2)} \to {\cal C}_{Gl_a(1)}$
and $k_\chi: {\cal C}_{Gl_a(1)} \to {\cal C}_{Gl_a(0)} =vec_\C$, attached to 
smooth characters $\rho, \chi$
of $k^*$ respectively,  appear on the left hand side of the next diagram. 
Notice, the product $\mu=\chi\rho$ of these two characters in fact defines a character $\mu$ of the center $Z$ of the linear group $Gl(2) \subseteq
Gl_a(2)$. For us, two exceptional choices $\mu=1$ and $\mu=\nu^2$ will play a crucial role, the first for the study of $H$-period functionals and the second for the study of exceptional poles of $L^{PS}(s,\Pi,\Lambda)$.

\bigskip
The Bessel quotients $\widetilde \Pi$ were investigated in \cite{RW}.
In most cases they define perfect modules in ${\cal C}_{Gl_a(1)}^{fin}$, hence their $(T,\chi)$-coinvariant spaces $\widetilde\Pi_{(T,\chi)}$ are one-dimensional vectorspaces for all smooth characters $\chi$ of $T$. It turns out that for $\chi=\rho$ they realize the $H$-period functionals $\ell$. Furthermore, for a certain
normalization of $\Pi$ by a character twist, we can assume $\rho=1$. Besides the
$H$-period functionals $\ell$ or their $Gl(2)$-linear resp. $Gl_a(1)$-linear avatars $\overline\ell: \overline \Pi \to (\C, \rho\circ\det)$, or $\widetilde\ell: \widetilde \Pi \to (\C, \rho\circ\det)$ as in the diagram, 
certain companion functionals $\widehat f: \widehat \Pi \to (\C, \nu\rho\circ\det)$ play an important role.
We show that the following diagram is commutative, where  $\widehat \Pi\in {\cal C}_{Gl(2)}^{fin}$ on the right side
is the central specialization $\widehat \Pi = (\overline\Pi\vert_{Gl(2)})_{(Z,\mu)} $  of the restriction of $\overline \Pi \in {\cal C}_{Gl_a(2)}$ to the subgroup $Gl(2)\subseteq Gl_a(2)$, i.e. the maximal quotient on which $Z$ acts by the character $\mu$. 
Assuming $\mu=\chi\rho$ we have

$$ 
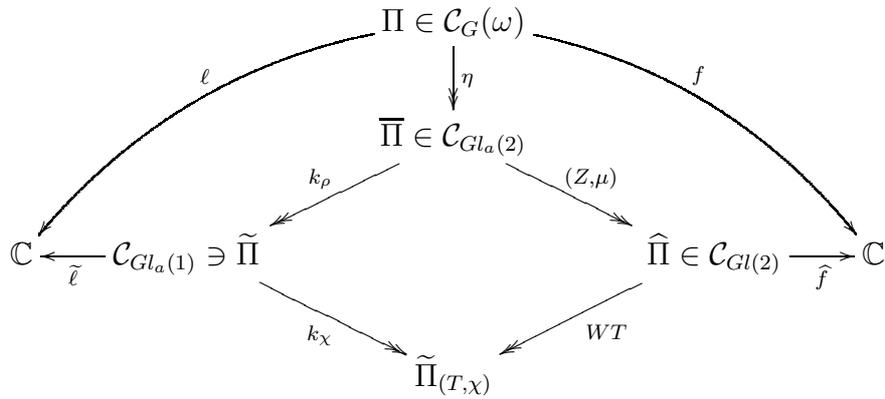
\xymatrix{ & & \Pi\in {\cal C}_G(\omega)\ar@/ ^1.7pc/[ddrr]^f \ar@/_1.7pc/[ddll]_\ell\ar@{->>}[d]^\eta & & \cr & &  \overline\Pi \in {\cal C}_{Gl_a(2)} \ar@{->>}[dl]_{k_\rho} \ar@{->>}[dr]^{(Z,\mu)}    &  &   \cr  \C   &
{\cal C}_{Gl_a(1)} \ni \widetilde \Pi \ \ \ \ \ \ar[l]^-{\widetilde\ell} \ar@{->>}[dr]_{k_\chi} &           &   \ \ \  \ \ \widehat \Pi \in {\cal C}_{Gl(2)} \ar@{->>}[dl]^{\ \ \ WT}\ar[r]_-{\widehat f} & \C\cr
&   &     \widetilde \Pi_{(T,\chi)} &  & \cr} $$
Roughly speaking, 
the classification of the $H$-period functionals is obtained from the left side 
of the diagram, for the particular choices $\rho=1, \chi=1$ and hence for $\mu=1$. Each $H$-period functional factorizes over certain functionals $\overline \ell$ resp. $\widetilde \ell$ on $\overline\Pi$ and $\widetilde\Pi$. We not only classify, but also construct $\ell$ from
functionals $\overline\ell$  by glueing them along $\widetilde \ell$.

\bigskip
The investigation of the exceptional poles of $L^{PS}(s,\Pi,\Lambda)$
is also related to the left side of the diagram. However, now this
leads to the choice of characters $\rho=1, \chi=\nu^2$ and hence $\mu=\nu^2$.
Finally everything boils down to determine the image of the paramodular new vector $v_{new}\in \Pi$ in the one-dimensional space $\widetilde \Pi_{(T,\chi)}$ for the special choice $\rho=1, \chi=\nu^2$. Since controlling the image of $v_{new}$ 
under the maps on the left side of the diagram is difficult, we move to the right side and modify $\rho$ and $\chi$, still leaving $\mu=\rho\chi$ constant.  The $Z$-specialization $\widehat\Pi$
of $\overline\Pi$, defined by the character $\mu$ of the center $Z$ of $Gl(2)$, remains the same for this modification. For the special choice $\rho=\nu$ and $\chi=\nu$ in this replacement, the $H$-period functionals on the extended Saito-Kurokawa representations $\Pi$ turn out to have some rather accessable companion
functionals $f : \widehat \Pi \to \C$. 
With the help of these companions $f$ we are able to determine the images of $v_{new}$ in $\widehat\Pi$
for $\mu=\nu^2$. Having this information at hand, we return to the relevant choice $\rho=1, \chi=\nu^2$
and complete the discussion, using the commutativity of the lower square of the diagram.
This utilizes a comparison isomorphism, obtained from the fact $\dim(\widetilde \Pi_{(T,\chi)})=1$. Here $WT$ denotes a map of Waldspurger-Tunnell type for the group $Gl(2)$, and eventually we also consider a first order deformation $WT'$ of it. Notice, in the cases where $\Pi$ is of extended Saito-Kurokawa type, $\mu=1$ and $\mu=\nu^2$ are precisely the choices where the representation $\widehat \Pi$ of $Gl(2)$ has one-dimensional irreducible constituents.

\bigskip
Both the $H$-period functionals $\ell$ and their companion functionals $f$ 
factorize over $\eta: \Pi \to \overline\Pi$ and are
equivariant with respect to a unipotent subgroup $\tilde N$ of $G$ and are equivariant for a certain character also  with respect to the maximal split torus of $G$. This character is implicitly given
by the central character $\omega$ of $\Pi$ and the characters $\rho$ and $\chi$ introduced above.
For representations of extended Saito-Kurokawa type the space of such $(\rho,\chi)$-functionals vanishes unless
 $\rho=1$ or $\chi=\nu$; see the remark following lemma \ref{dim1}.  This explains our choices for $\chi$ and $\rho$ if $\mu=\rho\chi$ is either $\mu=1$ or
$\mu=\nu^2$.
 
\bigskip
Wondering why the cases $\mu=1$ and $\mu=\nu^2$ play such a prominent role, the reader may find it helpful to look at the global picture. Poles of the global $L$-series occur at $s=\frac{3}{2}$, giving rise to global $H$-periods, and $s=-\frac{1}{2}$ only, and these two  points are related to each other under $s\mapsto 1-s$ by the global functional equation which switches the $L$-factors $L(s-\frac{1}{2})$ resp. $L(s+\frac{1}{2})$ of the global adelic $L$-function in the Saito-Kurokawa cases. The poles at $s=\frac{3}{2}$ resp. $s=-\frac{1}{2}$ arise from these two factors. Aside from the local normalization factor $\delta_P^{1/2}=\nu^{3/2}$, the characters $\mu=\delta_P^{1/2}\nu^{-3/2}$ and $\mu=\delta_P^{1/2}\nu^{1/2}$ define 
the two exceptional characters $\mu=1$ and $\mu=\nu^2$.
Locally there is also a functional equation for $L^{PS}(s,\Pi, \Lambda)$ relating the two exceptional points. Unfortunately, this seems  no help for computing the image of $v_{new}$ under the $(1,\nu^2)$-functionals via the corresponding image under the $(1,1)$-functionals $\ell$ (which are nonzero for trivial reasons). But this difficulty can be resolved by the study of the interpolating $(\nu,\nu)$-functionals $f$, whose particular nature also seems to be of independent interest. In these functionals $f$ the unique behaviour of $\Pi$ at the two exceptional \lq{points}\rq\ is still encoded, in an almost symmetric way, as will become clear later from the proof in the sections 6 and 7.

\section{$H$-period functionals}

\noindent
For a t.d. topological group $H$ let
${\cal C}_H$ be the category of smooth $H$-modules. The abelian categories ${\cal C}_H$ do have enough injectives, which allows to define $Ext$-groups by injective resolutions \cite{Borel_Wallach}.
For a character $\omega$ of the center of $H$ let ${\cal C}_H(\omega)$ denotes the category of smooth representations of $H$ on which the center of $H$ acts with the character $\omega$. We fix a prime element $\pi$ of $k^*$. Let $\nu $ denote the normalized
absolute value on $k^*$ normalized such that $\nu(\pi)= q^{-1}$ where $q$ is the cardinality of the residue field.

\bigskip
For the symplectic group of similitudes $G\! =\! GSp(4)$ in four variables over the
local field $k$, defined by  
$$G =\Bigl \{ g\in Gl(4,k) \ \Bigl \vert\ g' \begin{pmatrix} 0 & E \cr - E & 0 \end{pmatrix} g = \lambda_G(g) \cdot  \begin{pmatrix} 0 & E \cr - E & 0 \end{pmatrix}\Bigr \} \ ,$$ 
let $P= MN$ denote the Siegel parabolic subgroup
with Levi component $M$ and unipotent radical $N$ as in [RW]. Notice, $\lambda_G(g)$
defines the similitude character $\lambda_G$ of $G$.
The elements $n$ in $N$ are the upper block triangular matrices in $G$
and the elements $m$ in $M$ are
$$ m= t_\lambda \cdot m_A \ = \ \begin{pmatrix} A  & 0 \cr 0 & \lambda \cdot (A')^{-1} \end{pmatrix} \ $$
for  $(A,\lambda)\in Gl(2) \times Gl(1)$. 
We put $\tilde t= diag(t_1,t_2,t_2,t_1)$
and  $t_\lambda = diag(E,\lambda E)$ for $t_1,t_2,\lambda\in k^*$.
Let ${\omega}$ be a fixed character of $k^*$,
considered as a character of the center $Z_G$ of $G$ via $z\cdot id \mapsto {\omega}(z)$.
Let $\kappa\in M$ be the involutive element $m_A$ for $A=(\begin{smallmatrix}0 & 1\cr 1 & 0\end{smallmatrix})$. Let $T$ denote the one-dimensional torus generated by the matrices $x_\lambda=diag(\lambda E,E)$ for $\lambda\in k^*$ and $\tilde T$ the rank two torus generated by all elements $\tilde t$ for $t_1,t_2\in k^*$. 

\bigskip\noindent
\textbf{The subgroup $H$}. In the following we consider a subgroup of $G$ isomorphic to $Gl(2)$, defined by the fixed embedding
$$ \iota: \begin{pmatrix} a & b\cr   c & d\end{pmatrix} \quad \mapsto\quad  \begin{bmatrix} a & b\cr   c & d\end{bmatrix} := \begin{pmatrix} ad-bc & 0  & 0 & 0 \cr 0 & a  & 0 & b \cr
 0 &  0 & 1 & 0 \cr  0 & c  & 0 & d  \end{pmatrix} \ .$$
By abuse of notation we identify this subgroup with $Gl(2)$. So, if we write $Gl(2)\subseteq G$, it is understood
that we use the embedding chosen above. The subgroup thus defined is contained in the Levi subgroup $L$ of
the Klingen parabolic subgroup $Q$ of $G$, as chosen in \cite{RW}. Notice, $L$ is isomorphic
to $Gl(2)\times Z_G$ for the center  $Z_G$ of $G$.

\bigskip
For the involution $\kappa$, the conjugate subgroup
$Gl(2)^\kappa = \kappa Gl(2) \kappa^{-1}$ commutes with the subgroup $Gl(2)$ of $G$ defined above.  
It is easy to see that the subgroup $H$
generated by $Gl(2)$ and its conjugate $Gl(2)^\kappa$ in $G$ is isomorphic to the subgroup
of $Gl(2) \times Gl(2)$ of elements $h=(g_1,g_2)\in Gl(2)\times Gl(2)$ such that 
$\det(g_1)=\det(g_2)$. For $h=(g_1,g_2)$, viewed as an element of $G$ under the embedding $(int_\kappa\circ \iota) \times \iota$, 
we have $\lambda_G(h)=\det(g_1)=\det(g_2)$.
By definition, the involution $\kappa\in G$ normalizes $H$. For $a\in k$ we later consider
 the following matrices $s_a \in Gl(2)^\kappa \subseteq H$
$$ s_{a} \ = \ \begin{pmatrix} 1 & 0  & a & 0 \cr 0 & 1  & 0 & 0 \cr
 0 &  0 & 1 & 0 \cr  0 & 0  & 0 & 1  \end{pmatrix} \ $$
and the subgroup $\langle s_a, a\in k\rangle$ of $G$ generated by these. 
$H$ contains $T$ and $\tilde T$.

\bigskip\noindent
\textbf{Bessel data}. We consider \begin{enumerate} 
\item $\tilde N = N\cap H$, the group generated by the $s_a$ and $\kappa s_a\kappa^{-1}$
for $a\in k$,
\item $\tilde R :\! = \tilde T \tilde N $, the so called \textit{Bessel} group contained in $H$
\item $\Lambda: \tilde R \to \C^*$, a character trivial on $\tilde N$ with $\Lambda\vert_Z\!=\!{\omega}$, hence 
 $\Lambda(\tilde t) = \rho(\frac{t_1}{t_2}){\omega}(t_2)$ for some character $\rho$ of $k^*$ (split Bessel character) in the sense of \cite{RW}.
\end{enumerate}
If we replace $(\Pi,\Lambda)$ by a twist $(\mu\otimes \Pi, \mu\otimes\Lambda)$, in the twisted Bessel model  the data 
$\omega, \rho, \rho^\divideontimes$ are replaced by $\mu^2\omega, \mu\rho, \mu\rho^\divideontimes$ respectively.

\bigskip\noindent
\textbf{Extended Saito-Kurokawa Representations}. 
Irreducible smooth representation $\Pi$ of $G$ with a split Bessel model are of extendended Saito-Kurokawa type if they belong to the Saito-Kurokawa cases \nosf{IIb, Vbc, VIc, XIb} or are of type \nosf{VId} in the notations of [ST94], [RS07]. Locally they are the analog of the  $P$-CAP and $B$-CAP representations as in \cite{PS83}.

\medskip\noindent
For  an irreducible representation $\Pi$ of $G$ in ${\cal C}_G(\omega)$
and a smooth character  $\rho$ of $k^*$ consider $\C$-linear  maps $$\ell: \Pi \to \C$$ with the property that
for all $v\in \Pi$ and all $h\in H$ the following holds
$$   \ell\bigl(\Pi(h) v\bigl) \ = \ \rho\bigl(\lambda_G(h)\bigl) \cdot \ell\bigl(v\bigl) \ .$$
We call such functionals \textit{ $H$-period
functionals} with respect to $\rho$.
If $\ell $ is an $H$-period functional with respect to $\rho$, then
$\ell^\divideontimes(v):= \ell(\Pi(\kappa)v)$ also is an $H$-period functional with respect to $\rho$ since $\kappa$ normalizes $H$. 
Using the classification of irreducible representations as in \cite{Sally-Tadic}, our first main result will be

\begin{thm} \label{maintheorem}
For an irreducible admissible representation $\Pi$ of $G=GSp(4)$ over the nonarchimedean local field $k$ and a character $\rho$ of $k^*$ the dimension of the space of $H$-period functionals $\ell\in
Hom_H(\Pi, \rho\circ\lambda_G)$ is at most one. This space is nonzero if and only if $\Pi$ is one of the representations of extended Saito-Kurokawa type (cases \nosf{IIb, Vbc, VIc, XIb,VId}) or $\Pi$ is  one-dimensional. In all these cases there is a unique  character $$\rho=\rho(\Pi) \ ,$$ only depending on $\Pi$, for which a nontrivial $H$-period functional exists on $\Pi$.
\end{thm}

\begin{proof}
This theorem is obtained from the gluing theorem \ref{gling}, which reduces
the assertions of theorem \ref{maintheorem} to the later proposition \ref{mainprop} which will be
proven at the end of the next chapter.
\end{proof}

\noindent
\textbf{Central character condition}. Suppose $\ell$ is a nontrivial $H$-period functional on $\Pi$ and $\omega$ is the central
character of the  irreducible representation $\Pi$. Since $\lambda_G(t)=t^2$ for $t\in Z_G$ and since $Z_G$ is contained in $H$, the $H$-equivariance of $\ell$ implies $   \omega(t) = \rho^2(t)$. For $\rho^\divideontimes := \omega/\rho$
the condition $   \omega(t) = \rho^2(t)$ gives $ \rho= \rho^\divideontimes $ and $\rho$ is smooth. 
Hence a necessary condition for the existence of nontrivial $H$-period functionals $\ell$ on $\Pi$
with respect to $\rho$ is the condition $\rho=\rho^\divideontimes$. So in the following, we can make and therefore will make the

\medskip
\textbf{Assumption}: $\ \rho=\rho^\divideontimes$.

\smallskip
\textbf{Normalization}. Let $\sigma\otimes\Pi$ denote the twist of $\Pi$ with the one dimensional representation
$\sigma\circ \lambda_G$ defined by a smooth character $\sigma$ of $k^*$.
If the Jacquet module $J_P(\Pi)$ of $\Pi$ for the Siegel parabolic $P$ is not trivial, then as in \cite{RW} we write $ \Pi = \sigma\otimes \Pi_{norm}$ for some character $\sigma$, where $\Pi_{norm}$ is normalized. Up to isomorphism, the list of normalized representatives can be found in the second column 
of \cite{RW}, table 1. We also define $\nosf{Vc}_{norm} = \nosf{Vb}$ using the fact 
$\nosf{Vc} \cong \chi_0\otimes \nosf{Vb}$ (for some quadratic character $\chi_0$).  

\bigskip
Obviously for character twists
$$Hom_H(\sigma\otimes \Pi, (\sigma\rho)\!\circ\!\lambda_G) \cong Hom_H(\Pi, \rho\!\circ\!\lambda_G)\ .$$ 
Furthermore  $\rho(\sigma\otimes\Pi_{norm})=\sigma \rho(\Pi_{norm})$
holds for the character $\rho(\Pi)$ defined by theorem \ref{maintheorem}, due to its characterization. 
This allows us to assume that $\Pi =\Pi_{norm}$ is normalized in the sense above.
The reason for this normalization will be the following: We find out later in lemma \ref{B6} that
$$\rho(\Pi_{norm})=1 \ .$$ 
This being said, we proceed by introducing the gluing construction of theorem \ref{gling} that allows to prove theorem \ref{maintheorem} in section two.

\noindent
\textbf{The functional $\overline\ell$}.   For a representation $\Pi\in {\cal C}_G(\omega)$,
in \cite{Roberts-Schmidt}, \cite{RW}, 
the space of coinvariants $\overline \Pi= \Pi_{\langle s_a, a\in k\rangle}$ of $\Pi$ with respect to the  subgroup $\langle s_a, a\in k\rangle \cong k$ of the unipotent radical of the Siegel parabolic subgroup of $G$
was considered. The $\C$-vectorspace
$\overline \Pi$ is a smooth module under a certain subgroup of $G$ that is isomorphic to the  \textit{affine linear group} $Gl_a(2)$. Being the semi-direct product of the linear group $Gl(2)$ over $k$ and the group
of translations on the $2$-dimensional standard representation $k^2$ of $Gl(2)$, by definition $Gl_a(2)$ contains the linear group $Gl(2)$. Suitably chosen as a subgroup of $G$ as in [RW], section 4.2 this will be our fixed subgroup $Gl(2)$ of $G$. It is contained in $H$ and 
together with its conjugate under $\kappa$ it generates $H$. Let ${\cal C}_2={\cal C}_{Gl_a(2)}$ be the category of smooth representations of $Gl_a(2)$. The categories of smooth representations 
of  affine linear group were first studied by Gelfand and Kazhdan. The irreducible objects $M$ in these categories can be determined by Mackey's theory; see \cite{Bernstein-Zelevinsky}, \cite{Roberts-Schmidt}.

\bigskip
$H$ contains the subgroup $\langle s_a\rangle$ of $G$ 
and  $\rho(\lambda_G(s_a))=1$. Hence $H$-period functionals $\ell$ factorize over a functionals
$\overline \ell$ on the quotient space $\overline\Pi = \Pi_{\langle s_a\rangle}$, defined as the maximal quotient
vectorspace
of $\Pi$ on which
$\langle s_a\rangle$ acts trivially
$$   \xymatrix{ \Pi   \ar@{->>}[rr]\ar[dr]_\ell &    &  \overline\Pi\ar@{.>}[dl]^{\exists ! \ \overline\ell}  \cr
&   \C   & } \ ,$$
so that $$\  \overline\ell \ \in \ Hom_{Gl(2)}(\overline\Pi, \rho\!\circ\! \det) \ ,$$ i.e. for all matrices $[\begin{smallmatrix} a & b \cr c & d \end{smallmatrix}]$
in $Gl(2) \subseteq Gl_a(2)$ and for all $\overline v\in \overline\Pi$
$$  \overline\ell( \begin{bmatrix} a & b \cr c & d \end{bmatrix} \overline v)
= \rho(ad-bc) \cdot \overline\ell(\overline v)\ .$$

\bigskip
\textbf{Further factorizations}. To analyze $ Hom_{Gl(2)}(\overline\Pi, \rho\circ \det)$ we consider the Bessel modules
$\widetilde \Pi= \Pi_{\tilde R, \Lambda}$. Here $\Pi_{\tilde R, \Lambda}$ denotes the space of $\Lambda$-coinvariants of $\Pi$ for the subgroup $\tilde R$ with respect to the character $$\Lambda(\tilde t) = \omega(t_2)
\rho(\frac{t_1}{t_2})= \rho(t_1)\rho^\divideontimes(t_2)$$ of $\tilde R$ defining the Bessel datum. So it is the maximal quotient of $\Pi$ on which $\tilde R$ acts by the character $\Lambda$, where
by definition $\tilde R$ is the group generated by the torus $\tilde T$
and the vector group generated by the matrices $s_a, a\in k$ and $\kappa s_a\kappa^{-1}, s\in k$.
For more details we refer to \cite{RW}. By the central character condition $\rho^\divideontimes =\rho$
we necessarily have
$$ \Lambda(\tilde t) = \rho(t_1)\rho(t_2) = (\rho\!\circ\! \lambda_G)(\tilde t) \ .$$ 
Each of the subgroups $\{0\} \subseteq \langle s_a, a\in k\rangle
\subseteq \tilde R  \subseteq T\tilde R \subseteq H$  
of $G$ inherits a character by restricting  $\rho\!\circ\!\lambda_G$. This character is trivial on $\tilde N$, and
coincides with the Bessel character $\Lambda(\tilde t)$ on $\tilde T$. Furthermore
$(\rho\!\circ\!\lambda_G)(x_\lambda) =\rho(\lambda)$ on $T$. 
%
Hence the inductive nature of taking coinvariants
immediately gives surjective $\C$-linear maps
$$  \Pi \twoheadrightarrow \overline\Pi 
\twoheadrightarrow \widetilde \Pi \twoheadrightarrow  \widetilde \Pi_{(T,\rho)} \ .$$
The left map is $T\tilde T \cdot Gl(2)$-equivariant, the second/third map is $T$-equivariant.
Since $H$ contains $\langle s_a, a\in k\rangle$ and $T\tilde T$ and $Gl(2)$, the
quotient $\Pi \twoheadrightarrow \Pi_{(H, \rho\circ\lambda_G)}$ factorizes over the
quotient $\Pi 
\twoheadrightarrow \widetilde\Pi_{(T,\rho)}= \widetilde \Pi_\rho$. Hence

\goodbreak
\begin{lem} \label{Praegluing} There exist canonical embeddings $$  Hom_H(\Pi, \rho\!\circ\!\lambda_G) \subseteq Hom_{Gl(2)}(\overline\Pi, \rho\!\circ\!\det)$$
and 
$$ Hom_{Gl(2)}(\overline\Pi, \rho\!\circ\!\det)\ \subseteq\ Hom_\C(\widetilde\Pi_\rho, \C)  \ .$$ 
If $\Pi$ is an irreducible, non-generic representation, then $\dim(\widetilde\Pi_\rho)\leq 1$.
\end{lem}

\begin{proof}
Only the last assertion remains to be shown, and it follows from the more general assertion \cite{RW}, thm. 5.22. 
\end{proof}

\noindent
\textbf{Glueing}.
The composite of the two inclusions given in lemma \ref{Praegluing}) 
$$ Hom_H(\Pi, \rho\!\circ\!\lambda_G)  \ \subseteq \  Hom_\C(\widetilde\Pi_\rho, \C) $$
is $\kappa$-equivariant. Indeed
the involution $\kappa$ acts on Bessel characters, mapping $\Lambda(t_1,t_2)$ to $\Lambda^\divideontimes(t_1,t_2) = \Lambda(t_2,t_1)$, with $\Lambda^\divideontimes =\Lambda$ by $\rho=\rho^\divideontimes$.
Since $\kappa$  commutes with $T$ and $\kappa$
acts on $\widetilde\Pi_\rho$, we write $\ell\mapsto \ell^\divideontimes$ for
the induced action on the space $Hom_\C(\widetilde\Pi_\rho, \C)$ of linear maps. 
Since $H$ is generated by $Gl(2)$ and $\kappa Gl(2)\kappa^{-1}$, these remarks
imply

\begin{lem} \label{glueing} Considered as subspaces of $Hom_\C(\widetilde\Pi_\rho, \C)$, we 
obtain
$$  Hom_{H}(\Pi, \rho\!\circ\!\lambda_G) \ = \  Hom_{Gl(2)}(\overline\Pi, \rho\!\circ\!\det)^\divideontimes
\cap Hom_{Gl(2)}(\overline\Pi, \rho\!\circ\!\det) \ .$$
\end{lem}

By the later proposition \ref{mainprop} for all irreducible generic representations $\Pi$  
we have
$Hom_{Gl(2)}(\overline\Pi, \rho\circ\det) = 0$. 
With this information  at hand we may assume, without restriction of generality, that in the last lemma
\ref{glueing} the representation $\Pi$ is non-generic. Then from lemma \ref{Praegluing}
we obtain $$Hom_\C(\widetilde\Pi_\rho, \C) \leq 1 \ .$$  By dimension reasons, this implies that the action of $\divideontimes$ 
on $Hom_\C(\widetilde\Pi_\rho, \C)$  automatically 
stabilizes the subspace $Hom_{Gl(2)}(\overline\Pi, \rho\circ\det)$, i.e.:
$$ Hom_{Gl(2)}(\overline\Pi, \rho\circ\det)^\divideontimes
\ = \ Hom_{Gl(2)}(\overline\Pi, \rho\circ\det) \ .$$
Hence by lemma \ref{glueing} and proposition \ref{mainprop} we obtain the next theorem.

\begin{thm} \label{gling}
For irreducible representations $\Pi$ of $G$, suppose 
the space of $H$-period functionals $   Hom_{H}(\Pi, \rho\!\circ\!\lambda_G)$ is not zero. Then
$$   Hom_{H}(\Pi, \rho\!\circ\!\lambda_G) \ \cong \  Hom_{Gl(2)}(\overline\Pi, \rho\!\circ\!\det) \ $$
and this space is one-dimensional\footnote{The involution $\divideontimes$ then acts by a sign on this space. It is not difficult to show that $H$-period functionals are non-trivial on the paramodular new vector. This allows to compute this sign in terms of the eigenvalue of the Atkin-Lehner involution on the space of paramodular new vectors, which is known by \cite{Roberts-Schmidt}.}.
\end{thm}

This \lq{glueing theorem}\rq\ computes the $H$-period functionals in terms of the simpler space
$Hom_{Gl(2)}(\overline\Pi, \rho\!\circ\! \det)$. To analyze the latter we more generally study functionals in $$Hom_{Gl(2)}(M, \rho\!\circ\! \det)\ $$
for arbitrary modules $M\in {\cal C}_2^{fin}$.
Notice the following weaker, but related result

\begin{lem} \label{C1} Suppose $\Pi$ is an irreducible representation in ${\cal C}_G(\omega)$
and $\rho=\rho^\divideontimes$ holds. Let $i_*(\rho\circ \det)$ denote the unique character of $Gl_a(2)$
that restricts to $\rho\circ\det$ on the subgroup $Gl(2)$ of $Gl_a(2)$.
Then  $$ Hom_{Gl_a(2)}(\overline\Pi, i_*(\rho\!\circ\! \det)) = 0 $$
(morphisms in the category ${\cal C}_{Gl_a(2)}$ ) unless $\Pi$ is one-dimensional (case \nosf{IVd}).
\end{lem}

\begin{proof} We may assume that $\Pi$ is normalized. For a $Gl_a(2)$-linear homomorphism
$$ \overline\Pi \twoheadrightarrow (\C,\rho\!\circ\! \det)$$ the composite
with $\eta:\Pi \twoheadrightarrow  \overline\Pi$ is trivial on the unipotent radical $Rad(Q)$ of the Klingen
parabolic subgroup $Q$ of $G$, hence factorizes over the unnormalized 
Jaquet module $J_Q(\Pi)=\Pi_{Rad(Q)}$ for the Klingen paraobolic $Q$. So, for the proof 
it suffices to show $\mu\neq \rho$
 for all
irreducible one-dimensional constituents  $\pi=\mu\circ\det$ of $J_Q(\Pi)$ of dimension 1. 

\medskip
By table A.5 of \cite{Roberts-Schmidt}, the characters $\mu$ that occur in $J_Q(\Pi)$ are:
\begin{description}
\item[\quad \nosf{}] \nosf{IIIb}\quad $\mu=\nu\chi_1$ and $\mu=\nu$  (here $\omega=\chi_1$)
\item[\quad \nosf{}] \nosf{IVc}\quad $\mu= \nu^2$ 
\item[\quad \nosf{}] \nosf{IVd}\quad $\mu = 1$ where $\dim(\Pi)=1$
\item[\quad \nosf{}] \nosf{VIc}\quad $\mu=\nu$
\item[\quad \nosf{}] \nosf{VId}\quad $\mu= \nu$.
\end{description}
For normalized $\Pi$ in case \nosf{IIIb} we have $\omega = \chi_1$ and $\chi_1 \neq \nu^{\pm 2}$. 
By our assumption $\rho=\rho^\divideontimes$, this implies $\rho^2 = \omega=\chi_1$. 
If $\rho=\mu$, this would imply  $\chi_1 = \rho^2 \in \{ \nu^2\chi_1^2, \nu^2\}$
and hence $\chi_1 = \nu^{\pm 2}$. A contradiction. 
Since  $\omega=1$ holds for normalized representations in this list except
for case \nosf{IIIb}, this implies $\rho^2 =\omega=1$ and rules out all other cases
where $\dim(\Pi)>1$.
\end{proof}

\goodbreak 
\section{ The classification of the $H$-period functionals}


\bigskip\noindent
\textbf{Compact induction}. For a metrizable t.d.-group $G$
a left invariant Haar measure $dg$ on $C_c^\infty(G)$ exists and is unique
up to a constant. The modulus $\Delta_G$ is defined by  $\int_G f(g)dg =
\Delta_G(g_0) \int_G f(gg_0)dg$ for all $f\in C_c^\infty(G)$. For a closed subgroup $\varphi: H \hookrightarrow G$
put $\delta(h) = \frac{\Delta_G(h)}{\Delta_H(h)}$ for $h\in H$ (the inverse to \cite{C}). 
The Hecke algebra $C_c^\infty(G)$ (as convolution algebra) 
 is a smooth $G$-module under left translation $L_{g_0}f(g)=f(g_0^{-1}g)$,
for $f\in C_c^\infty(G)$. For a smooth $G$-module $(V,\pi)$ the action of $G$ extends to a natural action of the Hecke algebra $C_c^\infty(G)$ on $V$. This defines endomorphisms $\pi(f): V\to V$ for $f\in C_c^\infty(G)$ such that $\pi(g) \pi(f)= \pi(L_gf)$ holds for all $f\in C_c^\infty(G)$ and $g\in G$. 

\bigskip\noindent
For smooth $(V,\pi)\in {\cal C}_G$ the space $(V,\pi)^{-\infty} = Hom_G(C_c^\infty(G), \pi)$ is a left $G$-module via the action of $G$ on $C_c^\infty(G)$ defined by $R_{g_0}f(g)= f(gg_0) \Delta_G(g_0)$. The
map $V \ni v \mapsto (T_v: f \mapsto \pi(f)v) \in V^{-\infty}$ is $G$-equivariant and  injective since
$(V,\pi)$ is smooth. Its image  is the subspace of all smooth vectors of $(V,\pi)^{-\infty}$, in other words: $(V,\pi)= ((V,\pi)^{-\infty})^\infty$. For an arbitrary  $G$-module $(W,\Pi)$ the subspace $W^\infty \subseteq W$ of smooth vectors
is stable under $G$ and defines the smooth subrepresentation $(W,\Pi)^\infty$ of $(W,\Pi)$. 
%
For
$\mu \in {\cal C}_H$ the unnormalized compact induced representation $ind_H^G(\mu)$ is a smooth representation
of $G$. For any representation $(V,\pi)$ of $G$ let $\varphi^*(\pi)$ denote the restriction $\varphi^*(\pi)(h)= \pi(\varphi(h))$ to a closed subgroup $\varphi: H\hookrightarrow G$. If $\pi$ is smooth, $\varphi^*(\pi)$ is smooth. 
 
\begin{lem} \label{Frobenius}
For a metrizable t.d.-group $G$ and a closed subgroup $\varphi: H \hookrightarrow G$ and $\mu\in {\cal C}_H$ and $\pi\in {\cal C}_G$, Frobenius reciprocity gives $$ Hom_G(
ind_H^G(\mu\otimes \delta), \pi) \cong Hom_H(\mu, \varphi^!(\pi))\ ,  $$ where $\varphi^!(\pi) = (\varphi^*(\pi^{-\infty}))^\infty$.
If $\dim(\pi)=1$, then $\varphi^!(\pi) = \varphi^*(\pi)$.
\end{lem}

\begin{proof}
The first assertion follows from \cite{C}, theorem 1.4 and  formula (33). For the second assertion we refer to
\cite{Bump}, proposition 4.3.2 which for $\dim(\pi)=1$ implies $\dim((\C,\pi)^{-\infty})=1$, and therefore $(\C,\pi)=(\C,\pi)^{-\infty}$.
\end{proof}

\noindent
\textbf{Digression on irreducible modules $M$}.  Up to isomorphism the irreducible modules of the abelian category
${\cal C}_2$ are $M=\mathbb S_2$, $M=j_!i_*(\mu)$ and $M=i_*(\pi)$ by \cite{Bernstein-Zelevinsky}.  Here $\mu$ is a smooth character of $k^*$ resp. $\pi$ an irreducible smooth
representation of $Gl(2,k)$. For  the notation see
\cite{RW}. We now compute $Hom_{Gl(2)}(M, \rho\circ \det)$ for
these modules $M$.

\begin{lem} \label{B2}\label{B3}\label{B4}For irreducible $M\in {\cal C}_2$, and irreducible representations $\pi\in {\cal C}_{Gl(2)}$
 and smooth characters $\mu, \rho$ of $k^*$ the dimensions of 
$Hom_{Gl(2)}(M\vert_{Gl(2)}, \rho\circ\det)$ are 
\begin{enumerate}
\item zero for $M={\mathbb S}_2$,
\item zero for $M=j_!i_*(\mu)$ except for $\rho=\nu^{-1}\mu$, where the dimension is 1,
\item zero for $M=i_*(\pi)$ and $\pi\in {\cal C}_{Gl(2)}$ except if $\pi\cong \rho\circ\det$, where the dimension is 1.
\end{enumerate}
\end{lem}

\begin{proof}
The $Gl_a(2)$-module $\mathbb S_2$, restricted to the subgroup
$Gl(2)$, is isomorphic to the Gelfand-Graev representation $ind_U^{Gl(2)}(\psi_{gen})$.
Since $Gl(2)$ and $U$ are unimodular, lemma \ref{Frobenius} implies
$$ Hom_{Gl(2)}(\mathbb S_2, \rho\!\circ\! \det) = Hom_U(\psi_{gen}, \rho\!\circ\! \det)\ .$$
Since $\det$ is trivial on $U$ and $\psi_{gen}$ is nontrivial on $U$, this gives the first claim.
For $M= j_!i_*(\mu)$ and a smooth character $\mu$ of $Gl(1)$ the module
$j_!i_*(\mu)$ 
is an irreducible module in ${\cal C}_2$ which,
restricted to $Gl(2)$, becomes  isomorphic to $ind_{\Gamma}^{Gl(2)}(\mu)$ for the mirabolic subgroup $\Gamma=\{h=(\begin{smallmatrix} a & b \cr 0 & 1
\end{smallmatrix})\vert a\in k^*, b\in k \} $ in $Gl(2)$. Since $\Delta_\Gamma(h) = \vert a\vert^{-1}$
in this case, lemma \ref{Frobenius} implies
$$ Hom_{Gl(2)}(j_!i_*(\mu), \rho\!\circ\! \det) = Hom_\Gamma((\nu^{-1}\mu)\!\circ\!\det, \rho\!\circ\! \det)\ .$$ Hence $  Hom_{Gl(2)}(j_!i_*(\mu), \rho\circ \det)$ has dimension zero except for $\nu^{-1}\mu=\rho$
where the dimension is $1$. This proves the second claim. The third claim is obvious.
\end{proof}

\noindent  
\textbf{Exact sequences}. For the remaining case $\pi=\rho\circ \det$ of 3. in lemma \ref{B2} we now turn to the study of $ Hom_{Gl(2)}(\overline\Pi, \rho\!\circ\! \det)$
for an irreducible module $\Pi \in {\cal C}_G(\omega)$.
We assume that the characters $\rho$ of $k^*$ satisfy the central character condition $\rho=\rho^\divideontimes$
and  that $\overline \Pi$ is defined by $\Pi$ as explained in section one.

\bigskip\noindent
First notice that $\overline\Pi$ contains a $Gl_a(2)$-submodule which is isomorphic to $\mathbb S_2^{m_{\Pi}}$. Here $m_\Pi$ is the dimension of the space of Whittaker models for $\Pi$. Furthermore, the 
quotient $Q$ of $\overline \Pi$ by this submodule sits in an extension (in the category ${\cal C}_2$)  
$$ \xymatrix{  0 \ar[r] &    j_!(A) \ar[r]^-{a} &  Q \ar[r]^-{b} &  i_*(B) \ar[r] & 0 } $$ 
where $A:= i_*(J_P(\Pi)_{\psi})\in {\cal C}_1$ and $B:=J_Q(\Pi)\in {\cal C}_{Gl(2)}$.
For this statement and  the notation used we refer to \cite{RW}, proof of
cor. 4.12. Now apply lemma \ref{B2}:

\medskip\noindent
a) Since $Hom_{Gl(2)}(\mathbb S_2^{m_{\Pi}}\vert_{Gl(2)}, \rho\circ\det)=0$  by lemma \ref{B2}, right exactness of the $Hom$-functor shows
$$ Hom_{Gl(2)}(\overline\Pi, \rho\!\circ\! \det)= Hom_{Gl(2)}(Q, \rho\!\circ\! \det)
\ .$$ So we may replace $\overline \Pi$ by $Q$.   

\medskip\noindent
b) Restricting the  $Gl_a(2)$-module $Q$  to 
$Gl(2)\subseteq Gl_a(2)$,  we next make use of the fact that for all irreducible representations $\Pi\in {\cal C}_G(\omega)$ of dimension $>1$
and all characters $\rho=\rho^\divideontimes$ the following holds:
$$Hom_{Gl(2)}(B, \rho\circ\det)=0\ .$$ 
Indeed, since by composition with $\overline\Pi \twoheadrightarrow Q \twoheadrightarrow i_*(B)$ any $Gl(2)$-linear functional
$\ell: B \to \rho\circ\det$  extends to a $Gl_a(2)$-linear
morphism $\ell: \overline\Pi \to i_*(\rho\circ\det)$ in ${\cal C}_2$, this assertion follows from lemma \ref{C1}. So,
again by the right exactness of the $Hom$-functor, we obtain the exact sequence
$$    0\! 
\!\to\! Hom_{Gl(2)}(Q, \rho\!\circ\!\det) \!\to\! Hom_{Gl(2)}(j_!(A), \rho\!\circ\!\det)  \!\to\! Ext^1_{{\cal C}_{Gl(2)}}(B, \rho\!\circ\!\det) \ . $$
c) Hence from step a)  we obtain

\begin{cor} \label{B5}
For irreducible $\Pi\in {\cal C}_G(\omega)$ of dimension $\dim(\Pi)>1$
the space $Hom_{Gl(2)}(\overline\Pi, \rho\!\circ\!\det)$ is isomorphic to the kernel
of the connecting morphism
$$  \delta: Hom_{Gl(2)}(j_!(A), \rho\!\circ\!\det) \!\to\! Ext^1_{{\cal C}_{Gl(2)}}(B, \rho\!\circ\!\det) \ .$$
\end{cor} 

This corollary will be further explored in proposition \ref{mainprop} by a detailed study 
of the spaces $Hom_{Gl(2)}(j_!(A), \rho\!\circ\!\det)$ and $ Ext^1_{{\cal C}_{Gl(2)}}(B, \rho\!\circ\!\det)$.

\medskip\noindent
We compute $Hom_{Gl(2)}(j_!(A), \rho\circ\det)$ in the next lemma \ref{B6}.
For this we examine the irreducible constituents $M=j_!i_*(\mu)$ of $\overline\Pi$. By \cite{RW}, 4.12
these are defined by characters $\mu= \nu^{3/2} \chi_{norm}$ that are associated to the irreducible constituents
$\chi_{norm}$ of the $T$-module $\delta_P^{-1/2}\otimes J_P(\Pi)_\psi$. 
The corresponding characters $\chi_{norm}$ of $k^*$ are listed in the column $\Delta_0=\Delta_0(\Pi)$ of  \cite{RW}, table 3. In this table
$\Pi=\Pi_{norm}$ is assumed, i.e. $\Pi$ is normalized. See also appendix A.3 of loc. cit. In \cite{RW}
we identified characters of $T$ 
with characters $\chi$ of $k^*$ by the convention $\chi(x_\lambda):=\chi(\lambda)$.
Now $\rho=\nu^{-1}\mu =\nu^{1/2}\chi_{norm}$ is a necessary condition for $\mu$ by lemma \ref{B2}. Since $\chi_{norm} \in \Delta_0(\Pi)$,
therefore $Hom_{Gl(2)}(j_!(A), \rho\circ\det) =0$ unless 
$$ \rho \ \in \  \nu^{1/2} \Delta_0(\Pi) \ =: \ \Delta_+(\Pi)  \ .$$
If we study $H$-period functionals $\ell$ in $Hom_H(\Pi, \rho\circ \lambda)$, 
then also the central character condition $\rho=\rho^\divideontimes$  must be satisfied, and 
this implies the even stronger condition
 $$ \fbox{$ \rho=\rho^\divideontimes \ \in \  \Delta_+(\Pi) $} \ ,$$
which finally will imply the next lemma.

\begin{lem} \label{B6}
For an irreducible representation $\Pi\in {\cal C}_G(\omega)$  
the condition 
$ \rho=\rho^\divideontimes \ \in \  \Delta_+(\Pi) $ uniquely determines $\rho=\rho(\Pi)$ 
and can be satisfied  only in the cases \nosf{IVb} and  \nosf{VId} and 
the Saito-Kurokawa cases \nosf{IIb},  \nosf{Vbc},  \nosf{VIc} and \nosf{XIb}. In these cases
$Hom_{Gl(2)}(j_!(A), \rho\!\circ\!\det)$ is zero unless
$\rho=\rho(\Pi)$. For $\rho=\rho(\Pi)$ we obtain
$$\dim\bigl(Hom_{Gl(2)}(j_!(A), \rho\!\circ\!\det\bigr) = 1 \ .$$ 
If $\Pi=\Pi_{norm}$ is normalized, $A=i_*(\nu)$, $\rho(\Pi)=1$ 
and $\omega = 1$ must hold.
\end{lem}  

\begin{proof}  For irreducible $\Pi\in {\cal C}_G(\omega)$ the condition
$ \rho=\rho^\divideontimes \ \in \  \Delta_+(\Pi) $ implies that  the Siegel Jacquet module $J_P(\Pi)$ cannot be zero.
The list of $\Delta_+(\Pi)$ with this property can be found in \cite{RW}, table 1 and table 3 in the normalized case
$\Pi=\Pi_{norm}$; recall \nosf{Vc}$_{norm}$=\nosf{Vb}. By a character twist we can immediately reduce to this situation. Then a case by case check of \cite{RW}, table 1 and table 3 shows
that the conditions $\Pi=\Pi_{norm}$, $\rho^2=\omega$ and $\rho\in \nu^{1/2}\Delta_0(\Pi)$ imply
$\omega=1$. 

\bigskip\noindent
Indeed case \nosf{I} is excluded since $\chi_1^{\pm 1}\chi_2^{\pm 1}\neq \nu$,  \nosf{IIIab} is excluded since $\chi_1\neq \nu^{\pm 2}$ resp. $\chi_1\neq 1$ and \nosf{X} is excluded because $\omega_{cusp}\neq \nu^{\pm 1}$ and
the central character $\omega$ of $\Pi$ is trivial in all remaining cases. Then $\rho^2=\omega=1$ implies that $\rho$ must be a quadratic character. This excludes \nosf{IIa} since $\chi_1^2\neq \nu^{\pm 1}$ and also excludes 
the cases \nosf{IVacd, Vad, VIab, XIa}. It leaves the cases  \nosf{IVb}, \nosf{VId}, \nosf{IIb},  \nosf{Vb}, \nosf{VIc}, and 
\nosf{XIb} where $\Delta_+(\Pi)=\{1\}$, hence $\rho=1$. Then $\Delta_0(\Pi)=\nu^{-1/2}\Delta_+(\Pi)
=\{ \nu^{-1/2} \}$ and therefore $\chi=\nu^{3/2}\chi_{norm}$ for $\chi_{norm}\in \Delta_0(\Pi)$ implies
$A=i_*(\nu)$. Hence $\dim\bigl(Hom_{Gl(2)}(j_!(A), \rho\!\circ\!\det\bigr) =1 $ by lemma \ref{B2} since $\rho=1$ and $\mu=\nu$.
\end{proof}

\bigskip\noindent
\textbf{Discussion of $Ext^1_{{\cal C}_{Gl(2)}}(B, \rho\circ\det)$}.
To study the connecting map $\delta$ considered in corollary \ref{B5} it suffices to discuss 
the cases of irreducible representations $\Pi$ of $G$ for which $Hom_{Gl(2)}(j_!(A),\rho\!\circ\! \det)$ is nonzero;
see lemma \ref{B6}. We may also assume that $\Pi=\Pi_{norm}$ is normalized.

\begin{lem}\label{Atable} Suppose $\Pi$ is an irreducible normalized representation
of $G$ in ${\cal C}_G(\omega)$ of extended Saito-Kurokawa type or of type \nosf{IVb}.
Then  there exists an exact sequence in ${\cal C}_2$
$$ \xymatrix{  0 \ar[r] &  j_!i_*(\nu)  \ar[r]^-{a} & \overline\Pi \ar[r]^-{b} &  i_*(B) \ar[r] & 0 } $$ 
where $B\in {\cal C}_{Gl(2)}$ is a smooth representation of $Gl(2)$ of finite length.
$B$ has the following (a priori not necessarily semisimple) constituents, listed 
together with their central characters $\omega_i$ with respect to the center $Z$ of $Gl(2)$:
\goodbreak

\begin{description}
\item[\quad {}]  \nosf{IIb}\quad $\nu\otimes (\chi_1 \!\times\! \nu^{-1/2})$ and $\nu\otimes (\chi_1^{-1}\! \times\! \nu^{-1/2})$ with $\omega_1, \omega_2\in \{\nu^{3/2}\chi_1, \nu^{3/2}\chi_1^{-1}\}$   
\item[\quad {}] \nosf{Vb} \quad $\nu\otimes (\nu^{1/2}\chi_0 \!\times\! \nu^{-1/2})$ with $\omega_1=\nu^2\chi_0$
\item[\quad {}]  \nosf{VIc}\quad  $\nu\circ \det$ with $\omega_1=\nu^2$
\item[\quad {}] \nosf{XIb} \quad $\emptyset$
\item[\quad {}] \nosf{VId} \quad $\nu\circ \det \ \oplus\ (\nu^{1/2}\! \times\! \nu^{1/2})$ with $\omega_1=\nu^2$ resp. $\omega_2= \nu$
\item[\quad {}] \nosf{IVb}\quad $Sp \oplus (\nu^{5/2}\times\nu^{1/2}) $ with $\omega_1= 1$ resp. $\omega_2= \nu^3$.
\end{description}
In this list $\chi_0^2=1$ and $\chi_1^2 \neq \nu^{\pm 1}$ and $Sp$ denotes the Steinberg representation.
\end{lem}

\begin{proof}
These representations $\Pi$, described in   lemma \ref{B6}, are all non-generic.  
Hence $m_\Pi=0$ and $\overline \Pi = Q$.  Therefore $\overline \Pi$ is an extension of $i_*(B)$ and $j_!(A)=j_!(\nu)$ by lemma \ref{B6}.
The constituents of $B$ for the representations $\Pi$ can be found in  $\sigma$ in \cite{Roberts-Schmidt}, table A.5.
Since $\Pi$ is normalized, in the notations of  \cite{Roberts-Schmidt}, table A.5 this means $\chi=\chi_1, \sigma=\chi_1^{-1}$  for case \nosf{IIb} and $\sigma=1$ in the other cases. This immediately gives our assertions.
\end{proof}

\smallskip

\begin{lem} \label{B7} Suppose $\rho=\rho^\divideontimes$ and
$\Pi\in {\cal C}_G(\omega)$ is irreducible  representation as in lemma \ref{B6}. Then  the group  $Ext^1_{{\cal C}_{Gl(2)}}(B, \rho\circ \det)$ vanishes except for the case \nosf{IVb}, where
 $  \delta: Hom_{Gl(2)}(j_!(A), \rho\!\circ\!\det) \to Ext^1_{{\cal C}_{Gl(2)}}(B, \rho\!\circ\!\det) $
 is injective.
\end{lem}  

\begin{proof} We use the well known fact that $Ext^1_{{\cal C}_{Gl(2)}}(M,N)$ vanishes
if $M,N$ are smooth representations of $Gl(2)$ on which the center $Z$ of $Gl(2)$ acts by different central characters
$\omega_M \neq \omega_N$; see \cite{Borel_Wallach}, p.14-15. 
Using decalage we apply this for  the irreducible constituents $M$
of $B$ and $N=\rho\circ\det$. 

\medskip\noindent
Since all normalized irreducible representations of lemma \ref{B6}
have trivial central character  $\omega=1$, the condition $\rho=\rho^\divideontimes$ implies $\rho^2=1$.
Hence $Z$ acts trivially on $N$, whereas all irreducible constituents $M$ of $B$, as listed in lemma \ref{Atable}, have central character\footnote{Later we are interested in $Gl(2)$-representations with exceptional central character $\mu=\nu^2$. Only
for \nosf{VIc, VId} the constituents of $B$ do have this central character since $\chi_0\neq 1$ and $\chi_1\neq \nu^{\pm 3/2}$ always holds in the cases above.
In the two cases \nosf{VIc, VId} the relevant constituent with the exceptional central character $\mu=\nu^2$
 is the one-dimensional direct summand $\nu\circ\det$ of $B$.}
 different from 1 (for \nosf{IIb} always $\chi_1 \neq \nu^{\pm 1/2}$, $\chi_1\neq \nu^{\pm 3/2}$ holds) except in the case where $\Pi$ is of type \nosf{IVb}. 
 
So it only remains to show  that $\delta$ is injective in the case where $\Pi$ is of type \nosf{IVb}. If this were not true,
$\Pi$ would admit a nontrivial $H$-period functional (cor. \ref{B5} and thm. \ref{gling})
which is impossible by lemma \ref{impossible}.
\end{proof}

\begin{prop} \label{mainprop}
For irreducible representations $\Pi$ in ${\cal C}_G(\omega)$ and $\rho=\rho^\divideontimes$
$$\dim\bigl( Hom_{Gl(2)}(\overline\Pi, \rho\!\circ\!\det)\bigr) \leq 1 \ ,$$ 
and this dimension is nonzero if and only if $\Pi$ is one-dimensional 
or belongs to the local Saito-Kurokawa
representations $\Pi$ of extended type (i.e. the cases \nosf{IIb},  \nosf{Vbc}, \nosf{VIc}, \nosf{XIb} and \nosf{IVd})
and  furthermore $\rho=\rho(\Pi)$ holds. If $\Pi=\Pi_{norm}$ is normalized, then $\rho(\Pi)=1$.
Finally $\rho(\sigma\otimes \Pi) = \sigma \rho(\Pi)$ holds.
\end{prop}

\begin{proof} 
Cor. \ref{B5} describes $Hom_{Gl(2)}(\overline\Pi, \rho\!\circ\!\det)\bigr)$ by
the connecting morphisms $$  \delta: Hom_{Gl(2)}(j_!(A), \rho\!\circ\!\det) \to Ext^1_{{\cal C}_{Gl(2)}}(B, \rho\!\circ\!\det) \ .$$
Lemma \ref{B6}  lists the cases where $Hom_{Gl(2)}(j_!(A), \rho\circ\det)\neq 0$, i.e. the cases for $\Pi$ listed in  prop. 
\ref{mainprop} up to \nosf{IVb}.  But $\delta=0$ holds except for the case \nosf{IVb}, where $\delta$ is injective.
This was shown in lemma \ref{B7}. The remaining assertions, including $\rho(\Pi_{norm})=1$, now immediately 
follow from lemma \ref{B6}.
\end{proof}

\bigskip\noindent

\goodbreak
\centerline{\textbf{Appendix on Klingen induced representations}}

\medskip\noindent
For the Klingen parabolic subgroup $Q$ of $G$
the induced representations $Ind_Q^G(\chi'\boxtimes \tau)$
are defined by 
$$  (\chi'\boxtimes \tau) \begin{pmatrix} t & * & * & * \cr
0 & a & * & b \cr 0 & 0 & * & 0 \cr 0 & c & * & d 
\end{pmatrix} \ = \ \chi'(t) \tau\begin{pmatrix} a & b \cr c & d 
\end{pmatrix} \ ,$$
normalized by the additional factor $\delta_Q^{1/2}= \vert t^2/(ad-bc)\vert$.
The central character of the induced representation is $\omega=\chi'\omega_\tau$.
The restriction of $\delta^{1/2}_Q \chi'\boxtimes\tau$ to our fixed subgroup
$Gl(2) \subseteq Q$ is $\nu\chi'\otimes \tau$.

%

\begin{lem}\label{impossible}
For irreducible representations $\Pi$ of type \nosf{IVb} there do not exist
nontrivial $H$-period functionals $\ell: \Pi \to \rho\circ\lambda_G$. 
\end{lem} 

\begin{proof} We may assume that $\Pi$ is normalized. Then a necessary condition for the existence of $\ell$ is
 $\rho=1$ (lemma \ref{gling}  and lemma \ref{B6}). $\Pi$ is a quotient of the Klingen induced representation
 $I\! =\! Ind_Q^G(\chi'\boxtimes\tau)$ for $\chi'\!=\! \nu^2$ and $\tau\!=\! Sp(\nu^{-1})$ by \cite{Roberts-Schmidt}, (2.9). So $\nu\chi'\otimes \tau = Sp(\nu^2)$. Since $H$-period functionals 
are equivariant with respect to the group $H^{ext}$ generated by $\kappa$ and $H$, for the proof
 it suffices that $I$ does not admit $H^{ext}$-period functionals. It is easy to show $G\! =\! \bigcup_m Q m H^{ext} 
 $ for $m\!=\! 1$ resp. for $m\! =\! m_A$, where $A$ is a unipotent lower triangular matrix. So it suffices that there are no nontrivial $H^{ext}$-functionals $ind_{ m^{-1}Qm \cap H^{ext}}^{H^{ext}}( \delta^{1/2}_Q \chi'\boxtimes\tau)\to  1$, or  by lemma \ref{Frobenius}  alternatively 
$$Hom_{\Gamma}((\delta^{1/2}_Q \chi'\boxtimes\tau)\vert_\Gamma,\delta)=0\  $$
for $\delta$ attached to $\Gamma\!:\! =\! m^{-1}Qm \cap H \hookrightarrow H$.
For both $m\! =\! 1,m_A$ the group $m\Gamma m^{-1}$
contains all matrices $[ \begin{smallmatrix} * & * \cr 0 & 1  \end{smallmatrix}]$. If $\Pi$ admits a $H$-period functional,
then $Sp(\nu^2)$ admits a nontrivial $Gl_a(1)$-equivariant functional $Sp(\nu^2)\!\to\! (\C,\delta)$. It would as a $Gl_a(1)$-morphism $(\C,\nu^3)\!\to\! (\C, \delta)$ factorize over the Jacquet module of $Sp(\nu^2)$ which is impossible 
since $\delta\vert_{Gl_a(1)}\!=\! \nu$ or $\nu^2$ (for $m\!=\! 1$ resp. $m\! =\! m_A$). 
\end{proof}

\noindent
\centerline{\textbf{Appendix on Extensions}}

\medskip
Let ${\cal C}_n$ denote the categories ${\cal C}_{Gl_a(n)}$
for the affine linear groups $Gl_a(n)$ over $k$.
${\cal C}_n$ contains $j_!({\cal C}_{n-1})$ as a full subcategory
that is equivalent to ${\cal C}_{n-1}$ \cite{Bernstein-Zelevinsky} so that $j_!({\cal C}_{n-1})$ is closed under
extensions. [Indeed, for an extension $0\to j_!(M_1) \to E \to j_!(M_2) \to 0$ in ${\cal C}_n$
we obtain $i^*(E)=0$ from $i^*j_!$ and the exactness of $i^*$. Hence $E=j_!j^!(E)$,
and the extension comes from ${\cal C}_{n-1}$ by applying the functor $j_!$].
So, to compute $Ext^1$-group it inductively suffices to understand the extensions in $ Ext^1_{{\cal C}_n}(i_*(\pi), M)$ for $M\in {\cal C}_n$; or even $M\in j_!({\cal C}_{n-1})$ by \cite{RW}, lemma 4.1 since $i_*({\cal C}_{Gl(n)})$ is also
closed under extensions. 

\begin{lem}\label{Restrict-extensions}
For $\pi\in {\cal C}_{Gl(n)}(\omega)$ and $M\in j_!({\cal C}_{n-1})$ the natural map
$$  Ext^1_{{\cal C}_n}(i_*(\pi), M) \longrightarrow Ext^1_{{\cal C}_{Gl(n)}}(\pi,
M\vert_{Gl(n)})$$
is injective.  \end{lem}

\begin{proof}
If an extension $0\to j_!(M)\to E\to i_*(\pi)\to 0$ in ${\cal C}_n$
admits a $Gl(n)$-linear splitting $s: \pi \to E\vert_{Gl(n)}$, we now show that the affine part
of the affine linear group $Gl_a(n)$ acts trivially on the image of $s$ and hence $s$ defines
a splitting in ${\cal C}_n$. Indeed, 
the image of $s$ is an eigenspace of the center $Z$ of $Gl(n)$ with respect to the central  character $\omega$ of $\pi$.
Every vector $v$ in the $s$-image is fixed by a small open subgroup $U$ of the affine part of $Gl_a(n)$ by smoothness of the representation $E$.
Since conjugation by $z\in Z$ acts on the affine part of $Gl_a(n)$ by rescaling, $zxz^{-1}\in U$ holds for some $z\in Z$ and hence $v=zxz^{-1}v$. Since $z^{-1}v = \omega(z)^{-1}\cdot v$, this implies $xv=v$ for all
$x$ in the affine part of $Gl_a(n)$. Therefore $s$ defines a $Gl_a(n)$-splitting of $E$.
\end{proof}

\section{The Piatetskii-Shapiro $L$-function}

\textbf{Exceptional poles}.  
 For an irreducible representation $\Pi\in {\cal C}_G(\omega)$ and a Bessel model of $\Pi$ attached to some Bessel character $\Lambda$  a local $L$-function $L^{PS}(s,\Pi,\Lambda)$ was defined by   Piatetskii-Shapiro as a product
$$  L^{PS}(s,\Pi,\Lambda) \ = \
L^{PS}_{ex}(s,\Pi,\Lambda)  L^{PS}_{reg}(s,\Pi,\Lambda) \ .$$
More generally $L^{PS}(s,\Pi,\mu,\Lambda)$,
for smooth characters $\mu$ of $k^*$, was defined in \cite{PS-L-Factor_GSp4}. These new local $L$-factors coincide with $L^{PS}(s,\mu\otimes\Pi, \mu\otimes\Lambda)$. Recall that $\mu\otimes\Lambda(\tilde t)=
\mu(t_1t_2)\Lambda(\tilde t)$, which amounts to replace $\rho\mapsto \mu\rho, \rho^\divideontimes \mapsto \mu\rho^\divideontimes$. So we may assume
$\mu=1$ if we drop the assumption that $\Pi$ is normalized.

In the following we discuss the extraordinary $L$-factor $L^{PS}_{ex}(s,\Pi,\Lambda)$
for split Bessel models with the Bessel character $\rho$ that describes $\Lambda$.

\begin{lem} \label{C19} For an irreducible representation $\Pi\in {\cal C}_G(\omega)$
and a split Bessel model for $\Lambda(\tilde t)=\rho(t_1)\rho^\divideontimes(t_2)$  
the $L$-factor $L^{PS}_{ex}(s,\Pi,\Lambda) $ divides
$$  L(s+\frac{1}{2}, \Lambda) = L(\rho, s+\frac{1}{2}) 
L(\rho^\divideontimes, s+\frac{1}{2}) \ .$$
\end{lem}

\begin{proof}
By the definition in
\cite{PS-L-Factor_GSp4}, section 4 and prop. 18 of \cite{Danisman3} the local 
$L$-functions $L^{PS}(s,\Pi,\mu,\Lambda)$ are the regularizing $L$-factors
of the following local zeta functions (with $t\overline t= t_1t_2$ in our notation)
$$ Z^{PS}(s,W_v,\Phi,\mu) =  \int_{k^*} \Bigl(\int_{K_H} \mu(det(k))\ell(\Pi((\begin{smallmatrix} xE & 0 \cr 0 & E   \end{smallmatrix})k) v)\ \mu(x)\vert x\vert^{s-\frac{3}{2}} d^*x$$ $$ \times \int_{k^*\times k^*} \Phi((0,t)k) \Lambda(t) (\mu\nu^{s+\frac{1}{2}})( t \overline t) d^*t \Bigr)dk    $$
for $W_v=W_v(g)=\ell(\Pi(g) v)$ running over all functions in the Bessel model,
i.e. for a fixed Bessel functional $\ell:\Pi\to \C$ running over all $v\in \Pi$.
Here $K_H$ is the maximal compact subgroup of $H\subseteq Gl(2)^2$ that stabilizes the lattice $({\frak o}\oplus {\frak o})^2$. Furthermore 
$\Phi$ is an arbitrary function in $C_c^\infty((k\oplus k)^2)$. As in \cite{Danisman}, prop 2.5 
one can also consider the regularizing $L$-factor of all 
$$ Z_{reg}^{PS}(s,W_v,\mu) =  \int_{k^*} \ell(\Pi(\begin{smallmatrix} xE & 0 \cr 0 & E   \end{smallmatrix}) v)\ \mu(x)\vert x\vert^{s-\frac{3}{2}} d^*x \ ,$$
again for $W_v$ running over functions in the Bessel model. This simplified $L$-factor is
the Euler factor $L(\mu \otimes M,s)$ of the $TS$-module
$M=\nu^{-3/2} \otimes (\beta_\rho(\Pi))/(\beta_\rho(\Pi)^S)$ attached to the Bessel model $M$
of $\Pi$ in \cite{RW}. The factors $L(\mu \otimes M,s)$ were computed
in [RW]. There also occurs an additional subregular and an additional exceptional local $L$-factor
in the Piateskii-Shapiro local $L$-factor $L^{PS}(s,\Pi,\Lambda)=L^{PS}_{ex}(s,\Pi\Lambda)L^{PS}_{reg}(s,\Pi,\Lambda)$ where
$L^{PS}_{reg}(s,\Pi,\Lambda)$ is the regularizing factor  of all zeta integrals
$Z^{PS}(s,W_v,\Phi,\mu)$ for $\Phi$ with the property $\Phi(0,0)=0$. The exceptional
$L$-factor $L^{PS}_{ex}(s,\Pi,\Lambda)$ then takes care of additional poles that arise
if we also allow   
$\Phi(0,0)\neq 0$. Hence to compute $L^{PS}_{ex}(s,\Pi,\Lambda)$, it suffices to consider for 
$\Phi$ the characteristic function $\Phi_0$ of the sublattice ${\frak o}^2\oplus {\frak o}^2  \subseteq k^2\oplus k^2$. 
Since $\Phi_0$ is $K_H$-invariant, one obtains 
$$ \int_{k^*\times k^*} \Phi_0((0,t)) \Lambda(t) (\mu\vert^{s+\frac{1}{2}})( t \overline t) d^*t = L(\mu\rho, s+\frac{1}{2}) 
L(\mu\rho^\divideontimes, s+\frac{1}{2})
   \  $$ if $\mu\rho$ and $\mu\rho^\divideontimes$ are unramified and otherwise the integral  is zero,
 so the zeta integral $ Z^{PS}(s,W_v,\Phi_0,\mu)$ simplifies and becomes 
$$ Z^{PS}(s,W_v,\Phi_0,\mu) =L(\mu\rho, s+\frac{1}{2}) 
L(\mu\rho^\divideontimes, s+\frac{1}{2})  Z_{reg}^{PS}(s,W_v^{av},\mu)$$
for the averages
$ W_v^{av}(g) =
\int_{K_H} \mu(det(k))\ell(\Pi((\begin{smallmatrix} xE & 0 \cr 0 & E   \end{smallmatrix})k) v)dk $ under $K_H$. Now assume $\mu=1$, which is achieved by passing from $\Pi$ to $\mu\otimes\Pi$. 
Since $\mu\rho$ and $\mu\rho^\divideontimes$ can be assumed to be unramified, $\mu\otimes\Pi$
is a twist $\sigma\otimes\Pi_{norm}$ of a normalized representation $\Pi_{norm}$ by an unramified character $\sigma=\mu\rho$. Since twisting by an unramified character is essentially the same as shifting the variable $s$,
this implies that from now on we may assume without restriction of generality that $\Pi$ is normalized and that $\sigma=1$. Then the $W_v^{av}(g)$ are nothing but the $ W_v(g)$ attached to the vectors $v\in  \Pi^{K_H}$, i.e. the vectors $v$ in the representation space of $\Pi$ that are invariant under the compact subgroup $K_H$. 
So obviously, $L^{PS}_{ex}(s,\Pi,\Lambda)$ is a divisor of  
$L(\rho, s+\frac{1}{2}) 
L(\rho^\divideontimes, s+\frac{1}{2})$.
\end{proof}
By the last remarks we see that  from now on we may tacitly assume that $\Pi$ is normalized and $\mu=1$ and we have seen

\begin{lem}\label{spherical} $L^{PS}_{ex}(s,\Pi,\Lambda)$ divides  the regularizing $L$-factor
${\cal L}(s)$ of all the functions $L(\rho, s+\frac{1}{2}) L(\rho^\divideontimes, s+\frac{1}{2})  Z_{reg}^{PS}(s,W_v,1)$,
where $v$ runs over the $K_H$-spherical vectors of the representation $\Pi$.
\end{lem} 

\noindent
\textbf{$H$-period functionals arising from poles}. 
As already observed by Piatetskii-Shapiro \cite{PS-L-Factor_GSp4} thm.4.2, see also \cite{Danisman3}, a pole of the exceptional $L$-factor $L^{PS}_{ex}(s,\Pi,\Lambda)$ at some point $s_0$ gives rise to
 a nontrivial $H$-period functional
$$   \ell:  \Pi  \longrightarrow  \C  $$
so that for  $h \in H$  (recall $\Pi$ is normalized and $\mu=1$)
$$  \ell( \Pi(h) v) \ = \  \vert det(h) \vert^{-(s_0+\frac{1}{2})} \cdot \ell( v) $$
holds,  where $det(h)$ for $h=(g_1,g_2)$ is $\lambda_G(h)= \det(g_1)=\det(g_2)$. 
In \cite{PS-L-Factor_GSp4} this is stated for simple poles only.
But $L(\rho, s+\frac{1}{2}) 
L(\rho^\divideontimes, s+\frac{1}{2})$ a priori has a double pole $s_0$ for $\rho=\rho^\divideontimes$,
hence the following remarks are in order: 
If $L^{PS}_{ex}(s,\Pi,\Lambda)$ has a pole at $s_0$, then independent of the pole-order at $s_0$ 
$$  \ell_0(v) \ := \ \lim_{s\to s_0} \frac{Z^{PS}(s,W_v,\Phi,1)}{L^{PS}(s,\Pi,\Lambda)} \quad , \quad \Phi=\Phi_0 $$
is a well defined nontrivial functional $\ell_0: \Pi \to \C$ (this is clear from the definition of $L^{PS}(s,\Pi,\Lambda)$ as a regularizing factor), and $\ell_0$ furthermore defines an $H$-period functional. Since the functional $\ell_0$ does not change if we replace $\Phi_0$ by any Schwartz function $\Phi$ with $\Phi(0)=1$, the latter assertion easily follows from the fact that the origin of $(k\oplus k)^2$ is a fixed point of $H$. Indeed, $L_{reg}^{PS}(s,\Pi,\Lambda)$ 
by definition regularizes all $Z^{PS}(s,W_v,\Phi,1)$ with $\Phi(0)=0$, hence in the limit 
$s\to s_0$ a pole of $L^{PS}_{ex}(s,\Pi,\Lambda)$ at $s_0$ annihilates the contributions of all $Z^{PS}(s,W_v,\Phi,1)$ with $\Phi(0)=0$.  
%
%
%

\bigskip
Since $s_0$ is a pole of $L(\rho,s + \frac{1}{2})
L(\rho^\divideontimes,s + \frac{1}{2})$, hence $  \rho\nu^{s_0+\frac{1}{2}} = 1$  or $ \rho^\divideontimes \nu^{s_0+\frac{1}{2}} = 1$.

\begin{prop} \label{pole->H}
If for a split Bessel model with Bessel character $\Lambda$ the local $L$-function $L^{PS}(s,\Pi,\Lambda)$ admits
an exceptional factor $L^{PS}_{ex}(s,\Pi,\Lambda)\neq 1$, then $\rho=\rho^\divideontimes$ and there exists a nontrivial $H$-period functional $\ell$ such that 
$   \ell: \Pi  \longrightarrow  \C 
$
such that $\ell(\Pi(h)v) = \rho(\lambda(h)) \ell(v)$  holds 
where $\rho, \rho^\divideontimes$ are defined by $\Lambda(t_1,t_2)=\rho(t_1)\rho^\divideontimes(t_2)$. Furthermore $\Pi$ is one of the extended Saito-Kurokawa cases \nosf{IIb, Vb, VIc, XIb} and  \nosf{VId} and is obtained from a normalized representation by an unramified character twist. Without restriction of generality $\Pi$ is normalized, and then
$\rho=1$, i.e. $\Lambda=1$ and $\omega=1$.  \end{prop}

\begin{proof}
As already explained, for $\rho$ or $\rho^\divideontimes$ such an $H$-period functional
does exist in the presence of a nontrivial exceptional $L$-factor for the Bessel model attached to $\Lambda$.
If such a $H$-period functional exists, $\rho^\divideontimes = \rho$ follows
from the central character condition and $\rho$ must be unramified. So the proposition
follows from theorem \ref{maintheorem}.  
\end{proof}

\begin{lem} \label{disjoint} For split Bessel models  
exceptional poles $s_0$ of $L_{ex}(s,\Pi,\Lambda)$
are disjoint from the set of poles of $L(\mu \otimes M,s)$.  
If  $\Pi$ is normalized, then the exceptional poles are located at $s_0=-1/2$.  
\end{lem}

\begin{proof}
To show this we may assume
that $\Pi=\Pi_{norm}$ is normalized. Then, by exchanging $\rho$  and 
$\rho^\divideontimes$ if necessary, we can assume $\rho=1$ for the split Bessel character $\rho$ in the situation of the last
proposition. So any exceptional poles of $L_{ex}(s,\Pi,\Lambda)$
are poles of $L(\nu^{1/2},s)L(\omega \nu^{1/2},s)$. If an exceptional pole exists, 
we have seen that an $H$-period functional for $\Pi$ exists. Hence $\omega=1$ must hold (lemma \ref{B6}).
Therefore $L(\nu^{1/2},s)L(\omega \nu^{1/2},s)= L(\nu^{1/2},s)^2$.
So $s_0=-1/2$ can be the only exceptional pole in the normalized case. 
The disjointness of $s_0$ to the set of poles of $L(\mu \otimes M,s)$, that is claimed in the lemma, now immediately follows from [RW], table 5 and $\chi_1^2\neq \nu^{\pm 1}$ and $\chi_0^2=1$ ([RS], table 2) where the regular poles are listed for normalized $\Pi$. 
\end{proof}

\noindent
\textbf{$K_H$-spherical vectors}. Suppose $L_{ex}^{PS}(s,\Pi,\Lambda)\neq 1$.
Then  the $L$-function has an exceptional pole and by prop. \ref{pole->H}  the irreducible representation $\Pi$ is  of type  \nosf{IIb, Vb, VIc, XIb, VId}. Without restriction of generality we can assume that $\Pi$ is normalized and $\rho=\omega=1$.  This implies that $\Pi$ is paramodular \cite{Roberts-Schmidt}, so $\Pi$ contains a new vector $v_{new}$ which is unique up to a constant.

\bigskip
By lemma \ref{disjoint}, $s_0=-1/2$  
is not a pole of $L(\mu \otimes M,s)$.  By lemma \ref{spherical},  $Z^{PS}_{ex}(s,\Pi,\Lambda)$ 
divides the (auxiliary) regularizing $L$-factor ${\cal L}(s)$ of $$ L(s+\frac{1}{2})^2 \cdot  Z_{reg}^{PS}(s,W_v,1) $$ for all the Bessel functions $W_v$ of the
$K_H$-spherical vectors $v$ in $\Pi$. This means,
we have to compute $Z_{reg}^{PS}(s,W_v,1)$ the Bessel functions $W_v$ of the $K_H$-spherical vectors $v$
in $\Pi$. The Bessel functions $W_v$ are determined by the images of the vectors $v$ of $\Pi$
in the Bessel model $\widetilde \Pi$ of $\Pi$, i.e. the images of $v$ under the projection 
$\Pi \twoheadrightarrow \widetilde \Pi $.
As in \cite{RW}, 4.2 this projection factorizes over the quotient 
$$\eta: \Pi \twoheadrightarrow \overline\Pi \ $$
which is a $Gl(2)$-module.
The center $Z$ of $Gl(2) \subseteq G$ consists of elements $$z_\lambda=[\begin{smallmatrix} \lambda & 0 \cr 0 & \lambda
\end{smallmatrix}] = x_\lambda \cdot \tilde t$$ for $\tilde t = diag(\lambda, 1, 1, \lambda)\in \widetilde T$
and $x_\lambda\in T$. We define  $\tau:=z_\pi^{-1}$ and  consider the unramified central character
$  \mu: Z \to \C^* $ defined by $\mu(\tau) = q^2$, i.e.
$$  \mu = \nu^2   \ .$$ 
Since the action of $\tau$ and $x_\pi^{-1} \in T$  on the Bessel model $\widetilde\Pi=\beta_\rho(\Pi)$
 coincides for the relevant Bessel character $\rho=1$,  for all $P\in \C[\tau,\tau^{-1}]$ 
we obtain  $$W_{\Pi(\tau) v}(x) = W_v(x/\pi)\ .$$
Therefore  
$ Z_{reg}^{PS}(s,W_{P(\Pi(\tau))v},1) = P(\nu(\pi)^{s-3/2}) Z_{reg}^{PS}(s,W_v,1)$. 
To study exceptional poles we need  to compute the order of  $Z_{reg}^{PS}(s,W_{P(\Pi(\tau))v},1)$ at $s_0=-1/2$. Evaluation at $s=s_0$ amounts to replace $P(\tau)$ by $P(q^2)$. 
From lemma \ref{KH} in the next section we obtain the following description of $\eta(\Pi^{K_H})$. 

\begin{lem}\label{KH-sp} Suppose the level $N$ of the new vector $v_{new}$ of $\Pi$ 
is $N=2m-\varepsilon$ 
 for $\varepsilon\in \{0,1\}$, $m\in\mathbb \N$.
Then
$\eta(\Pi^{K_H}) = \C[\tau^{-1}]\, \tau^{-m} (q+\tau)^{\varepsilon} \, \eta(v_{new})$.
\end{lem}

Since $\mu(\tau)^{-m} (q+\mu(\tau))^{\varepsilon}= q^{-2m}(q+q^2)^{\pm \varepsilon}\neq 0$,
lemma \ref{KH-sp} implies that  ${\cal L}(s)$ is completely determined
 by the order of $Z_{reg}^{PS}(s,W_{v_{new}},1)$ at $s=s_0$ for the single Bessel function $W_{v_{new}}$ attached to a new vector $v_{new}$. In the definition of ${\cal L}(s)$, instead of considering all
 $Z_{reg}^{PS}(s,W_v,1)$ for arbitrary $v\in \Pi^{K_H}$ it suffices to compute the single
 integral $Z_{reg}^{PS}(s,W_{v_{new}},1)$ at $s=s_0$ for 
$v=v_{new}$. This implies
 
\begin{cor} \label{Hilfsschritt}
Suppose $\Pi$ is normalized and $L_{ex}^{PS}(s,\Pi,\Lambda)$ has a pole.
Then the auxiliary $L$-factor ${\cal L}(s)$ is $L(s+\frac{1}{2})^k$ for $k=2,1,0$ 
depending on whether $ Z_{reg}^{PS}(s,W_{v_{new}},1)$ 
has a zero of order $0$, $1$ or $\geq 2$ at $s_0=-\frac{1}{2}$ for the paramodular
new vector $v_{new}$.
\end{cor}

In the situation of cor.\ref{Hilfsschritt} a Bessel model $\widetilde \Pi= k_\rho(\overline\Pi)$ of $\Pi$ always exists for $\rho=1$ and it
defines a perfect $Gl_a(1)$-module (\cite{RW}, cor.6.9) 
 with an exact sequence
 $$  0 \to \mathbb E[\nu]  \to  \widetilde\Pi \to  k_\rho(i_*(B)) \to 0 \ .$$
 By the perfectness of $\widetilde\Pi$, the space  $\widetilde\Pi_{T,\chi}$ of $(T,\chi)$-coinvariants has dimension 1 for every smooth character $\chi$ of $T$ and the quotient map  
$$ p_\chi: \widetilde\Pi \to \C \cong \widetilde\Pi_{T,\chi}  $$  
is described by the regularized zeta integrals $Z(f,\chi^{-1},s)$ at $s=0$ (\cite{RW}, 3.5). 
Moreover, for $\chi=\nu^2$ the regularized integrals 
$Z_{reg}^{PS}(s+s_0,W_{v},1)$ and the regularized integrals
$Z(W_v, \chi^{-1}, s)$ coincide. 
Since $s_0$ is not a pole of $L(\mu\otimes M,s)$, as follows from \cite{RW}, lemma 3.31, 
(see also \cite{RW}, lemma 3.34 concerning higher order vanishing), this implies 

\begin{lem}
$Z_{reg}^{PS}(s_0,W_{v},1) = c \cdot p_{\chi}(v)$ for $\chi=\nu^2 $ and some constant $c\neq 0$.
\end{lem}

%
%
%

\section{Paramodular vectors}
Let $K_G$ be the maximal compact subgroup
of unimodular matrices in $G$ and let 
$K_G({\frak p}^n) \subseteq K_G$ denote the principal congruence subgroups. 
We defined $H\subseteq G$ and its maximal compact subgroup
$K_H$. Let  $K_n$ be the subgroup of $K_G$ generated by $K_H $ and
$K_G({\frak p}^n)$. $K_n$ contains $K_G({\frak p}^n)$ as a normal subgroup
with finite quotient $K_n/K_G({\frak p}^n) \cong H({\frak o}/{\frak p}^n)$ so that 
$K_H$ is the intersection $K_H = \bigcap_{n\geq 0} K_n$.
For any smooth representation $\Pi$  of $G$ on a vectorspace $V$ 
therefore
$$   V^{K_H} \ = \ \bigcup_{n\geq 0}\  W(n) $$
is the increasing union $\cdots \subseteq W(n) \subseteq W(n+1) \subseteq \cdots $ 
of the subspaces  of fixed vectors
$ W(n) = V^{K_n} $ of $V$ under $K_n$.

\bigskip
To understand the subspaces $W(n)$ we will relate them
to the spaces of paramodular vectors in $\Pi$.
For this we give a short overview on the relevant results of \cite{Roberts-Schmidt}. 
Since in \cite{Roberts-Schmidt} the symplectic group as matrix group is written  using Witt's
rather than Siegel's conventions, most matrices and formulas in loc. cit. look different
from ours. For the convenience of the reader we therefore use  matrix conventions for the group $G$
as they are used in \cite{Roberts-Schmidt} for the rest of this section.

\bigskip
This being said, the element $\tau \in Gl(2) \subseteq H \subseteq G$ is then given by 
$$ \tau^{-1} := 
\begin{bmatrix} \pi & 0 \cr
0 & \pi \end{bmatrix} := \begin{pmatrix} \pi^2 & 0 & 0 & 0 \cr 
0& \pi & 0 & 0 \cr
0 & 0 & \pi & 0 \cr
0 & 0 & 0 & 1 \end{pmatrix} \ , $$ 
and  $\tau$ coincides  with $\eta$, as defined in \cite{Roberts-Schmidt}, form. (3.5),
modulo the center
of $G$. 
$$ \tau^{-1}  \begin{pmatrix} a &  u_1 &  u_2 & b \cr 
v_1& \alpha & \beta & u'_2 \cr v_2 & \gamma & \delta & u'_1 \cr
c & v'_2 & v'_1 & d \end{pmatrix} \tau  \ = \  
\begin{pmatrix} a & \pi u_1 & \pi  u_2 & \pi^{2}  b 
\cr  \pi^{-1}  v_1& \alpha & \beta & \pi u'_2 
\cr  \pi^{-1} v_2 & \gamma & \delta & \pi  u'_1 
\cr \pi^{-2} c & \pi^{-1}  v'_2 & \pi^{-1}  v'_1 & d \end{pmatrix} $$
implies
$$ \tau^{-n} \begin{pmatrix} {\frak o} & {\frak o} & {\frak o} & {\frak p}^{-2n} \cr 
{\frak p}^{2n} & {\frak o} & {\frak o} & {\frak o} \cr
{\frak p}^{2n}  & {\frak o} & {\frak o} & {\frak o} \cr
{\frak p}^{2n} & {\frak p}^{2n} & {\frak p}^{2n} & {\frak o} \end{pmatrix} 
\tau^{n} \ = \
\begin{pmatrix} {\frak o} & {\frak p}^n & {\frak p}^n & {\frak o} \cr 
{\frak p}^{n} & {\frak o} & {\frak o} & {\frak p}^n \cr
{\frak p}^{n}  & {\frak o} & {\frak o} & {\frak p}^n \cr
{\frak o} & {\frak p}^{n} & {\frak p}^{n} & {\frak o} \end{pmatrix} 
 \ .$$
Hence, in terms of the paramodular groups $K({\frak p}^{n})$, defined as in 
\cite{Roberts-Schmidt} p. 31,
we get $  \tau^{-n} K({\frak p}^{2n}) \tau^{n} =  K_n  $.
For a smooth representation $\Pi$  of $G_{ad}$ on a vectorspace $V$, 
let $V(n)= V^{K({\frak p}^n)}$ be the space of paramodular vectors in $V$
of level $n$. Then
$$  W(n) \ = \ \tau^{-n}(V(2n))  \ .$$
Indeed  $v\in V(2n)\Longleftrightarrow K({\frak p}^{2n}) v=v \Longleftrightarrow    \tau^{n} K_n \tau^{-n}v =v
 \Longleftrightarrow  K_n \tau^{-n}v = \tau^{-n}v$,  where we write $gv$ or $g(v)$ instead of
$\Pi(g)(v)$.

\bigskip
Since we are only interested in the local irreducible representations of extended Saito-Kurokawa type, we can make
the following  

\bigskip\noindent
\textbf{Assumption}. \textit{Suppose
$\Pi$ is an irreducible normalized representation of $G$ of extended Saito-Kurokawa type \nosf{IIb, Vb, Vc, VIc, VId, XIb}}.

\bigskip
Then, by a character twist, we may and therefore will assume that the central character of $\Pi$
is trivial;  \cite{Roberts-Schmidt}, table A.12 implies that a twist $\sigma\otimes \Pi_{norm}$
is paramodular for type \nosf{IIb, Vb, VIc, VId, XIb} if and only if $\sigma$ is unramified except for
{Vb}. Here the twist with unramified $\chi_0\sigma$ adds the cases {Vc}.  
This being said, we can normalize $\Pi$ by an unramified twist. 
We remark: If $\Pi$ is not paramodular, all $W(n)$ are zero and hence $\Pi^{K_H}=0$ 
since $V(2n) =\eta^{n}(W(n))$ 
consists of paramodular vectors, as explained below.
This allows to normalize $\Pi$ and to replace the group $G$ by its adjoint group $G_{ad}$.
An equality of matrices in $G$ for the rest of this section therefore, by abuse of notation, will be understood be an equality  in $G_{ad}$. 

\bigskip
Let $N$ be the level of a paramodular irreducible representation
$V$ of $G$. Then $V(N)$ is generated by $v_{new}$.
If $N=2m-\varepsilon$ for $\varepsilon\in\{0,1\}$, from the
level raising operator $\theta' = \tau + \sum_{a\in {\frak o}/{\frak p}} s_{a\pi^{-N-1}}$ (\cite{Roberts-Schmidt}, lemma 3.2.2 ii)) we obtain $\theta'(v_{new})\in V(N+1)$. Hence from \cite{Roberts-Schmidt}, prop. 5.5.13 and the old form principle of loc. cit. the following holds

\begin{enumerate}
\item $W(m) = \C \cdot \tau^{-m} v_{new}$ \ if $\varepsilon=0$,
\item $W(m) = \C \cdot \tau^{-m} \theta'(v_{new})$ \  if $\varepsilon=1$. 
\end{enumerate}

For the $Q$-linear quotient map 
$  \eta: \Pi \twoheadrightarrow  \overline\Pi $ notice 
$\eta(\tau^{-1} v) = [\begin{smallmatrix} \pi & 0 \cr 0 & \pi \end{smallmatrix}] \eta(v)$
and $\eta(s_a v)= \eta(v)$ for all $a\in k$. Hence we obtain $\eta(\theta'(v_{new}))= (\tau +q)
\eta(v_{new})$ in the case 2. above.

\begin{lem}\label{KH} $\dim W(n)= \min(m+1-n,0)$ holds
for some integer $m\geq 0$ defined by $2m =N-\varepsilon, \varepsilon\{0,1\}$ above. 
Furthermore the $Gl_a(2)$-linear quotient map 
$  \eta: \Pi \twoheadrightarrow  \overline\Pi $
is injective on $\Pi^{K_H}$ such that 
$$\eta(\Pi^{K_H}) = \bigoplus_{n\geq m}
[\begin{smallmatrix} \pi & 0 \cr 0 & \pi \end{smallmatrix} ]^n
\ \eta(v)$$ holds for a basis vector $v$ of $W(m)=\C\cdot v$, and
$\eta(v)$ in $\overline\Pi$ is invariant under the compact group $Gl(2,{\frak o}). ({\frak p}^m)^2 $ in $ Gl_a(2)$, but
not invariant under $Gl(2,{\frak o}). ({\frak p}^{m-1})^2$. 
\end{lem}

\bigskip
Lemma \ref{KH} almost immediately follows 
from the following list of properties 1.-5. if we take into account
$\eta(\Sigma v) = q^2[\begin{smallmatrix} \pi & 0 \cr 0 & \pi \end{smallmatrix}] \eta(v)$:
Under the assumption on $\Pi$ from above, 
there exists an integer $m\geq 0$ such that
\begin{enumerate}
\item $ W(n)= 0$ for $n< m$ and $\dim W(m)=1$.
\item $W(m) = \C v$ so that $v$ is invariant under
$$\begin{pmatrix} 1 & u & v & 0 \cr 
0& 1 & 0 & * \cr
0 & 0 & 1 & * \cr
0 & 0 & 0 & 1 \end{pmatrix}   \ .$$ 
for all $(u,v)\in ({\frak p}^m)^2$, but not all $(u,v)$  in $({\frak p}^{m-1})^2$.
\item $\dim(W(n)) = n+1-m $ for all $n\geq m$. 
\item $ W(n+1) = W(n) + \Sigma(W(n))$ holds for all $n$ where
 $$\Sigma = \sum_{a=1}^q  \begin{pmatrix} 1 & 0 & 0 & a \cr 
0& 1 & 0 & 0 \cr
0 & 0 & 1 & 0 \cr
0 & 0 & 0 & 1 \end{pmatrix} \tau^{-1} \cdot \sum_{a=1}^q  \begin{pmatrix} 1 & 0 & 0 & \frac{a}{\pi} \cr 
0& 1 & 0 & 0 \cr
0 & 0 & 1 & 0 \cr
0 & 0 & 0 & 1 \end{pmatrix} 
  \ .$$ 
\item $V(N)=\C \cdot v_{new}$ and $V(N+1)= \C \cdot
\theta'(v_{new})$ and $v = v_{new}$ or $v= \theta'(v_{new})$, depending on whether $\varepsilon = 0$ or $1$.

\end{enumerate} 

\begin{proof}  Under the assumption on $\Pi$ at the beginning of the section 
there exists an integer $N=N_{\Pi} \geq 0$ (level of $\Pi$) so that $M=2m$
and $N\leq M \leq N+1$ for an integer $m\geq 0$ such that the assertion made in the lemma holds.
Assertions 1. and 3. follow from \cite{Roberts-Schmidt}, theorem 5.6.1 i) and ii)
and table A.12. In particular $\Pi$ is paramodular. For Assertion 2. notice $\tau^m(v)\in V(2m)$, so the claim easily follows from 
\cite{Roberts-Schmidt}, lemma 3.2.4. Since $v$ and hence $\tau^m(v)$ is invariant under $K=Gl(2,{\frak o} \subseteq Gl(2)$ and since $K$ acts transitively on primitive vectors in $({\frak p}^{m-1})^2$ we can also extend the assertion (3.8) in loc. cit. on vectors of type $(u,v)=(u,0)$ for arbitrary vectors  $(u,v)\in ({\frak p}^{m-1})^2$.

\bigskip
By loc. cit, theorem 5.6.1 iii) all $V(n)$ are generated from
$V(N)$ by applying certain raising operators $\theta, \theta'$  together with $\tau$, 
and on $V(n)$ one has $$\theta' := \theta'_{n+1} := \tau +  \sum_a\begin{pmatrix} 1 & 0 & 0 & \frac{a}{\pi^{n+1}} \cr 
0& 1 & 0 & 0 \cr
0 & 0 & 1 & 0 \cr
0 & 0 & 0 & 1 \end{pmatrix} \ $$
 by formula (3.7) loc.cit.
 By \cite{Roberts-Schmidt}, proposition 5.5.13 our assumption on $\Pi$
implies $\theta=\theta'$. Hence $V(n)$ is generated by $\theta^d \tau^e V(N)$
for $d,e\geq 0$ and $d+2e = n - N$; see also the corresponding assertions of
prop. 5.5.12 for \nosf{VId}  and lemma 5.5.6 for the Saito-Kurokawa type representations $\Pi$. 
Table A.12 shows
$W(m)=V(N+1)=V(N)$ if $N$ is odd, and $W(m)=V(N)$ if $N$ is even.
Hence $$V(2n) = \bigoplus_{i=0}^{n-m} \ (\theta')^{2i} \tau^{n-i-m }   V(2m) \ .$$
$\tau^{n-m - i}   V(2m) \subseteq V(2n-2i)$ gives
$  (\theta')^{2i} =  \theta'_{2n-1} \cdots \theta'_{2n-2i+1} $,
and the formula $\tau^{-k} \theta'_{2k+\ell} \tau^{k} = \theta'_{\ell}$ implies
$$ (\theta')^{2i} \tau^{n-i-m }   V(2m) =   \tau^n (\theta'_0 \tau^{-1}\theta'_1)^i   \tau^{-m }   V(2m) \ .$$ 
Hence $  \tau^{-n}V(2n) = \bigoplus_{i=0}^{n-m} (\theta'_0 \eta^{-1}\theta'_1)^i  \tau^{-m} V(2m) $
resp.
$$ W(n) = \bigoplus_{i=0}^{n-m} \ (\theta'_0 \tau^{-1}\theta'_1)^i  \ W(m) $$
for all $n\geq m$. Notice $\theta'_0 \tau^{-1}\theta'_1= \theta'_0\theta'_{-1} \tau^{-1}$.
This proves the lemma. \end{proof}

\section{The companion functionals $f$}

For irreducible representations $\Pi$
of extended Saito-Kurokawa type and the specific pair of characters $\rho=\nu$ of $\tilde T$
and $\chi=\nu$ of $T$ we will now construct  
$\tilde R \cdot T$-linear auxiliary functionals called $(\rho,\chi)$-functionals  $$ f: \Pi \to (\C,\rho\boxtimes\chi)  \ ,$$
satisfying $f(\Pi(\tilde t x_\lambda)v)=\Lambda(\tilde t)\chi(\lambda) f(v)$ for all $\tilde t\in \tilde T$, $x_\lambda\in T$ and $v$ in $\Pi$.
The construction for the particular case \nosf{VId} will be different from the construction
in the Saito-Kurokawa cases that we discuss first.

\textbf{Saito-Kurokawa cases}. Each local Saito-Kurokawa
representation $\Pi$ of type \nosf{IIa, Vb, VIc, XIb} has a generic companion representation $\Pi_{gen}$ defined as follows:  Both
$\Pi$ and $\Pi_{gen}$ are the only constituents of a suitably chosen induced representation
$I=Ind_P^G(\pi\boxtimes \nu^{1/2})$ for $\pi\in {\cal C}_{Gl(2)}$ induced from the Siegel parabolic subgroup $P$ as in \cite{RW}, table 1 (here we use normalized induction and notation as in \cite{RW}) such that there exists an exact sequence
$$ \xymatrix{ 0 \ar[r] & \Pi \ar[r]^u & I \ar[r]^-{v} & \Pi_{gen} \ar[r] & 0} \ ,$$
dual to the situation defining $\Pi$ as a Langlands quotient. To be concrete:
We choose $\pi$ as
$\nu^{-1/2}\chi_1 \times \nu^{-1/2}\chi_1^{-1}$, $Sp(\nu^{-1/2}\chi_0)$, 
$Sp(\nu^{-1/2}) $ and $\nu^{-1/2}\otimes\pi_{cusp}$ in the four cases respectively. 
%
The representation space of $I$ is the space of smooth functions $f:G\to \sigma$
such that $g(pg)=\delta_P^{1/2}(p)(\pi\boxtimes\nu^{1/2})(p)f(g)$ holds for all $p\in P, g\in G$. $G$ acts on $I$ by right translation. The central character $\omega$ of $I$ is trivial, so $x_\lambda$ and $t_{\lambda}^{-1}$
act on $I$ in the same way.
The evaluation $eval: f(g)\mapsto f(e)$ at the unit element
defines a $P$-equivariant map 
$$   eval: I \ \twoheadrightarrow \ \delta_P^{1/2} \otimes (\pi\boxtimes \nu^{1/2}) \ .$$
The $P$-module on the right side is a trivial module of the unipotent radical $N$ of $P$.
The torus $T\subseteq P$ acts by the character $\chi(x_\lambda)=\nu(\lambda)$. Indeed,    
$\delta_P^{1/2}(x_\lambda)=\nu^{3/2}(\lambda)$ and
$x_\lambda$ acts on $\pi\boxtimes \nu^{1/2}$ by
multiplication with 
 $\nu^{1/2}(\lambda^{-1})=\nu^{-1/2}(\lambda)$.
 Since $eval$ is equivariant with respect
to the subgroups $TS$ and  $\tilde T \tilde N$ of $P$, we can apply
the right exact Bessel functor $\beta_\rho$ to the morphism $eval$.
By Waldspurger-Tunnell $\beta_\rho(\delta_P^{1/2} \otimes (\pi\boxtimes \nu^{1/2}))$
is a one-dimensional $\C$-vectorspace on which $T$ acts by the character $\nu$.
Hence for $\chi=\nu$ we obtain $TS$-linear maps in ${\cal C}={\cal C}_{Gl_a(1)}$:
$$         \beta_\rho(eval):  \beta_\rho(I) \ \twoheadrightarrow \  (\C,\chi) = i_*(\chi) \ .$$
$i_*(\chi) \in {\cal C}$ is trivial as an $S$-module, so the morphism 
$\beta_\nu(eval)$ factorizes over the  map $\widetilde I:=\beta_\rho(I) \twoheadrightarrow
\pi_0(\widetilde I)$ to the quotient space $\pi_0(\widetilde I)$ of $S$-coinvariants (for the notation see \cite{RW}).
Composed with the quotient map $I
\twoheadrightarrow \beta_\rho(I)$ this defines 
 a $TS$-linear surjections $$ I \twoheadrightarrow  \beta_\rho(I) \twoheadrightarrow  \pi_0( \beta_\rho(I))
\twoheadrightarrow  i_*(\nu)\ .$$ 
 For generic representations $\Pi_{gen}$ by \cite{RW}, prop. 6.3 we have  $\beta^\rho(\Pi_{gen})=0$. Hence \cite{RW}, lemma 4.15 implies that $   \beta_\rho(u)$ 
 is injective and 
this gives an exact sequence
$$ \xymatrix{ 0 \ar[r] & \beta_\rho(\Pi) \ar[r]^{\beta_\rho(u)} & \beta_\rho(I) \ar[r]^-{\beta_\rho(v)} & \beta_\rho(\Pi_{gen}) \ar[r] & 0} \ .$$
Since $\pi_0$ is an exact functor
on ${\cal C}_{Gl_a(1)}$,  lemma 4.15 of \cite{RW} gives 
the exact sequence of $T$-modules
$$ 0\to \pi_0(\widetilde \Pi) \to \pi_0(\widetilde I) \to \pi_0(\widetilde \Pi_{gen}) \to 0 \ .$$
The companion representations $\Pi_{gen}$ of the Saito-Kurokawa representations are of type \nosf{IIa, Va, VIa, XIa}. The normalized characters that occur in $\pi_0(\widetilde \Pi_{gen})$
are listed in the column $\widetilde\Delta$ of table 3, Multisets of \cite{RW}.
The $T$-character $\chi=\nu$ corresponds to the following normalized $T$-character $\chi_{norm}=\nu^{-3/2}\chi=\nu^{-1/2}$ by normalization, in the notations of loc.cit.  
By \cite{RW}, table 3 and the conditions listed in table 1 of loc. cit. the character
$\chi_{norm}=\nu^{-1/2}$ does not appear as constituent in the
relevant modules $\pi_0(\widetilde \Pi_{gen})$.  Since $T$ acts on $i_*(\nu)$
by the character $\chi=\nu$, therefore the last exact sequence 
shows that the 
map $\pi_0(\widetilde{eval}\circ \widetilde u):  \pi_0(\widetilde \Pi) \to i_*(\nu)$ is an isomorphism. 
Hence for $\rho=\nu$ the map
$\tilde f := \beta_\nu(eval \circ u)$   is surjective and $T$-linear and makes the following diagram commutative
$$   \xymatrix@+3mm{  \beta_\nu(\Pi) \ar@{^{(}->}[r]^{\beta_\nu(u)} \ar@{->>}[dr]_{\tilde f} &   \beta_\nu(I)  \ar@{->>}[d]^{\beta_\nu(eval)}    \cr
&   i_*(\nu). } \ $$
If we compose $\tilde f$
with $\Pi\twoheadrightarrow \overline\Pi \twoheadrightarrow \beta_\nu(\Pi)$, this 
yields functionals denoted $\overline f$ resp. $f$ on $\overline\Pi$ resp. $\Pi$.
The nontrivial functional $f:\Pi \to \C$ is $(\chi,\rho)$-equivariant with respect  to   
  the character $\chi\boxtimes\rho=\nu\boxtimes\nu$ of $\tilde T \times T$
so that 
$  f\in Hom_{\tilde R \cdot T}(\Pi, \nu\boxtimes \nu)$.

\bigskip
\textbf{The case  \nosf{VId}}. Now we discuss the irreducible representations 
of $G$ in the remaining case \textbf{VId}. In this case
the representation $I=Ind_Q^G(1)$ (normalized induction from the Klingen parabolic subgroup $Q$)  
decomposes into a direct sum $I \cong \Pi \oplus \Pi'$ of two irreducible representations; \cite{SLN}, p.129. One of the summands 
$\Pi'$ is normalized of type \nosf{VIc} and paramodular of level 1  \cite{Roberts-Schmidt}.
The other one is $\Pi$ which is normalized of type $\nosf{VId}$ and paramodular of level 0 (spherical). 
Evaluation at the unit element $e$, on the space of functions defining the induced representation $I$, now
gives a $Q$-linear map $I \twoheadrightarrow (\C, \delta_Q^{1/2})$. Composing this evaluation with the inclusion $\Pi\hookrightarrow I$, we obtain 
a $Q$-linear map $f: \Pi \to \C$. This map is equivariant on the radical of $Q$ and 
equivariant for $\tilde T\times T$ up to a twist by $\nu\boxtimes \nu$
 since the restriction of $\delta_Q^{1/2}$ to $Gl(2) \subseteq Q$
is the character $\nu\!\circ\!\det$. Furthermore, it is equivariant for the 
upper triangular matrices in $Gl(2)\subseteq Q$. In other words, $f$
defines an element in $Hom_{\tilde R \cdot T}(\Pi, \nu\boxtimes \nu)$.
To show that $f$ is nontrivial, we show $f(v_{new})\neq 0$ for the spherical  vector
$v_{new}\in \Pi$. Indeed $eval$ is nontrivial on the spherical vector of $I$ by 
Iwasawa decomposition, so the claim immediately follows from the explicit construction of $f$ in terms of $eval$.  Since $\Pi'$ is paramodular of level one, it does  not contain the spherical vector of $I$. Therefore $f\neq 0$, moreover $f(v_{new})\neq 0$ holds for the spherical vector $v_{new}$ in $\Pi$. In the next lemma we show that $f$ generates
$Hom_{\tilde R \cdot T}(\Pi, \nu\boxtimes \nu)$.


\begin{lem}\label{dim1}
 For the normalized irreducible representations $\Pi$ of extended Saito-Kurokawa type
 $ \dim( Hom_{\tilde R \cdot T}(\Pi, \nu\boxtimes \nu)) = 1 $ holds.    
\end{lem}

\begin{proof} 
By definition $Hom_{\tilde R \cdot T}(\Pi, \nu\boxtimes \nu)=Hom_\C(\beta_\rho(\Pi)_\chi, \C)$ for $\rho=\nu$, $\chi=\nu$ holds. It is known
that a normalized irreducible representation $\Pi$ of extended Saito-Kurokawa type
does not have a Bessel model for $\rho\neq 1$. Hence the Bessel module $\beta_\rho(\Pi)$ has degree zero
for $\rho\neq 1$,  so that $\beta_\rho(\Pi) \cong \pi_0(\beta_\rho(\Pi))$ holds. The semisimplification of the $T$-module
$\pi_0(\beta_\rho(\Pi))$ is described by the $L$-factor $L(\beta_\rho(\Pi),s)$.  
By \cite{RW}, prop. 4.25 
we know that
$$  L(\beta_\rho(\Pi),s) = L(J_P(\Pi)_\psi,s) L(\beta^\rho(\Pi),s) \ .$$ 
For $\rho\neq 1$, from \cite{RW}, prop. 6.3 we conclude $ \beta^\rho(\Pi)=0$ for  irreducible normalized  $\Pi\not\cong$ \nosf{Vd}. 
This gives $L(\beta_\rho(\Pi),s) = L(J_P(\Pi)_\psi,s)$ in our situation, and the characters of the $L$-factors of $L(J_P(\Pi)_\psi,s)$ 
are listed in the column $\Delta_0(\Pi)$ of  table 3 (multisets) in \cite{RW}, up to a shift by 
the normalization factor $\delta_P^{1/2}=\nu^{3/2}$. 
For the normalized irreducible representations $\Pi$ of extended Saito-Kurokawa type 
this table shows  that $L(\nu,s)$ appears with multiplicity
one in $L(J_P(\Pi)_\psi,s)$; indeed $\chi=\nu$ corresponds to the normalized character
$\chi_{norm}=\nu^{-1/2}$ in the column $\Delta_0(\Pi)$.
This completes the proof. It shows that  that $\beta_\rho(\Pi)_\chi$  has
dimension one for $\rho=\nu$ and $\chi=\nu$.
\end{proof}

\textbf{Remark}. The same argument also implies $ \dim( Hom_{\tilde R \cdot T}(\Pi, \rho\boxtimes \chi)) = 0 $ 
for all irreducible normalized representations $\Pi$ of extended Saito-Kurokawa type unless $\rho=1$ or $\chi=\nu$ holds. 

\begin{lem}\label{new=sur}
The normalized representations $\Pi$ of extended Saito-Kurokawa type \nosf{IIb, Vb, VIc, XIb} and \nosf{VId} are paramodular representations. The space
$$  Hom_{\tilde R \cdot T}(\Pi, \nu\boxtimes\nu) \ = \ \C \cdot f  \ $$
is spanned by a $(\nu,\nu)$-functional $f: \Pi \to \C$
which is nonzero on the paramodular new vector $v_{new}$ of the representation $\Pi$.
\end{lem}

\begin{proof} $  Hom_{\tilde R \cdot T}(\Pi, \nu\boxtimes\nu)$ has dimension one
by lemma \ref{dim1}. We explicitely constructed a generator $f$
for all cases. In the case \nosf{VId}, where the new vector $v_{new}$ is spherical, we already have seen $f(v_{new})\neq 0$. It remains to show $f(v_{new})\neq 0$  
for the Saito-Kurokawa cases. We use that $f$ was explicitely constructed from the evaluation $eval$ at $e$ on the induced representation $Ind_P^G(\pi\boxtimes \nu^{1/2})$. 
By \cite{Roberts-Schmidt}, page 171 formula (5.38), quotation of  thm. 5.2.2 and  prop. 5.5.5 part i) and ii) we obtain
$eval(v_{new}) =w_{new}$ for the new vector of  the representation $\delta_P^{1/2}\otimes  (\pi\boxtimes\nu^{1/2}) \in {\cal C}_{Gl(2)}$ we induce from.
Hence it suffices to know that new vector of the $Gl(2)$-representation  $\delta_P^{1/2}\otimes (\pi\boxtimes\nu^{1/2})$ projects to a nonzero vector in $\delta_P^{1/2}\widetilde(\pi\boxtimes\nu^{1/2}) \cong \C$. 
This projection is of Waldspurger-Tunnell type, so  
the desired property follows from the next lemma. 
\end{proof}

\begin{lem} \label{Waldspurger-Tunnell} For irreducible representations $\pi\in \C_{Gl(2)}$ of dimension  $>1$ and unramified characters
$\rho$ of $k^*$ the Waldspurger-Tunnell 
space of coinvariants $\pi_\rho$ of $\pi$, defined as the maximal quotient space of $\pi$ on which all matrices $diag(1,\delta), \delta\in k^*$ in $Gl(2)$ act
as $\rho(\delta)$, has dimension one. 
The quotient map $p: \pi\twoheadrightarrow \pi_\rho$
is nontrivial on the Atkin-Lehner new vector $w_{new}$ of $\pi$. 
The same holds for the spherical
vector $w_{sph}$ of the representation $\pi=\nu^{1/2}\times \nu^{-1/2}$. 
\end{lem}

\begin{proof} Replacing $\pi$ by a twist, we can assume $\rho=1$. 
The restriction of $\pi$ to $\{(\begin{smallmatrix} * & * \cr 0 & 1\end{smallmatrix})\} \subseteq Gl(2)$ defines a perfect $Gl_a(1)$-module $M$ of degree 1 in all cases\footnote{This would not be true for $\nu^{-1/2}\times \nu^{1/2}$.} considered
(\cite{RW}, lemma 3.20). Hence $M$ can be embedded in
$C_b^\infty(k^*)$, and this defines the zeta integrals $Z(g,s) =\int_{k^*} g(x)\vert x\vert^s d^* x$ for $g\in M$. The projection $p: \pi \to \pi_\rho$  is realized as $p(g) = \lim_{s\to 0} Z(g,s)/L(M,s)$, up to a nonvanishing constant that depends on the chosen embedding. 
For irreducible $\pi$  the $L$-factor $L(M,s)$ agrees with the usual definition
of $L(\pi,s)$ after a shift of variables $s\mapsto s-1/2$. The zeta integral $Z(g_{new},s)$, attached to the Whittaker function $g_{new}=W_{w_{new}}$ of the new vector $w_{new}$, is the $L$-factor of $\pi$  up to the constant $W_{new}(1)\neq 0$.  For irreducible representations $\pi$ this is well known (\cite{Roberts-Schmidt}, page 3). For the indecomposable representation
$\pi=\nu^{1/2}\times \nu^{-1/2}$ the same holds for the spherical vector $w_{sph}$.
In this case $L(M,s)\!=\! L(1,s)L(\nu,s)$. Indeed by \cite{Bump}, Exercise 4.6.2 
$Z(g_{sph},s)/g_{sph}(1)$ is obtained from $ \sum_{m\!=\! 0}^\infty t^m q^{-m/2} Tr(Sym^m(diag(\alpha_1,\alpha_2))$ by evaluation at $t\!=\! q^{-s}$. Here we have $\alpha_1\!=\! \nu^{1/2}(\pi)$ and  $\alpha_2\! =\! \nu^{-1/2}(\pi)$. Thus the $L$-factor $Z(g_{sph},s)$ becomes $(1- q^{-1/2}\alpha_1 t)^{-1} (1- q^{-1/2}\alpha_2 t)^{-1} g_{sph}(1)$ evaluated at 
$t=q^{-s}$. Hence $Z(g_{sph},s)= L(1,s)L(\nu,s) g_{sph}(1)$, which proves our claim.
\end{proof}

\textbf{The functionals $\tilde f$}. For the extended Saito-Kurokawa representations $\Pi$ 
let us come back to lemma \ref{Atable}. We have already shown $A= i_*(\nu)\in {\cal C}$,
so we obtain for all cases an the exact sequence in ${\cal C}_2$
$$ 0 \to j_!i_*(\nu) \to \overline\Pi \to  i_*(B) \to 0 \ .$$
We now apply the Bessel functor $\beta_\nu$ to this exact sequence.
For the Mellin transforms
of characters $ \rho\neq 1 $ (\cite{RW}, lemma 4.25) 
we know that
$M_\rho(i_*(\nu)=k_\rho(j_!i_*(\nu)) \cong i_*(\nu)$ and
$M^\rho(i_*(\nu))=k^\rho(j_!i_*(\nu)) =0$ holds . Since 
$\beta^\rho(\Pi)=0$ for $\rho\neq 1$ \cite{RW}, prop. 6.3 
for $\rho=\nu$ we therefore obtain the following
long exact sequence in ${\cal C}_T$
$$ \xymatrix{ 0 \ar[r] & k^\nu(i_*(B)) \ar[r]^\delta &  i_*(\nu)  \ar[r]^-{a} &  \beta_\nu(\Pi)\ar@{->>}[d]_{\tilde f}  \ar[r]^-{b} &  k_\nu(i_*(B)) \ar[r] & 0 \cr 
 &  &     & i_*(\nu) &  &  }\ ,$$
taking into account 
$\beta^\nu(\Pi)=k^\nu(\overline\Pi)$ and $\beta_\mu(\Pi)=k_\nu(\overline \Pi)$.
For $\rho=\nu$ there does not exist a Bessel model for $\Pi$. Hence $\beta_\nu(\Pi)$ and $\beta^\nu(\Pi)$ is a $ST$-module of degree 0. 
Hence all modules in this long exact sequence are trivial $S$-modules.
If we compute the multiplicity of the character $\chi=\nu$ of $T$ in the $T$-modules
$k_\nu(i_*(B))\cong k^\nu(i_*(B))$, we obtain the next lemma
which will give us another interpretation
of the functionals $\tilde f = \beta_\nu(f)$.

\begin{lem} \label{F2} For extended Saito-Kurokawa representations $\Pi$
and the Bessel character $\rho=\nu$ the following holds: 
\begin{enumerate}
\item $k_\nu(i_*(B))=0$ holds for $\Pi$ of type \nosf{IIb, Vb, Vc, XIb}  and the composition $\tilde f \circ a$ induces an isomorphism 
$$ \xymatrix{0 \ar[r] &  i_*(\nu)\ar[dr]_\sim  \ar[r]^-{a} &  \beta_\nu(\Pi)\ar@{->>}[d]_{\tilde f}  \ar[r]^-{b} &  k_\nu(i_*(B)) \ar[r] & 0 \cr 
  &     & i_*(\nu) &  &  }\ $$
\item For $\Pi$ of type \nosf{VIc}, \nosf{VId} on the other hand the map $a$ is zero
 is zero, hence $b$ 
 induces an isomorphism 
$\beta_\nu(\Pi)\cong k_\nu(i_*(B)) $ 
$$ b: \beta_\nu(\Pi) \cong i_*(\nu) $$ 
so that $\tilde f$ factorizes over the quotient map $b$,
inducing an isomorphism $k_\nu(i_*(B)) \cong i_*(\nu)$
$$ \xymatrix{ \ar[r] & i_*(\nu) \ar[r]^-{0} & \beta_\nu(\Pi)\ar[d]_{\tilde f}^\sim \ar[r]^b_\sim & i_*(\nu) \ar@{.>}[dl]^\sim \ar[r] & 0 \cr
& & i_*(\nu) &  &  }$$
 \end{enumerate}
\end{lem}

\begin{proof} For the proof it suffices that $k_\nu(i_*(B))\cong i_*(\nu)$ holds for \nosf{VIc, VId}
and that the $T$-character $\chi=\nu$ does not appear
as a constituent of $k_\nu(i_*(B))$ for \nosf{IIb, Vb, Vc, XIb}.
This easily follows from the table  in lemma \ref{Atable} that lists the constituents of $B$.
Indeed, only constituents of $B$ with the central character $\omega_i=\chi\rho =\nu^2$
can contribute, but these only exist in the cases \nosf{VIc} and \nosf{VId}.
\end{proof}

\section{Specialization of $\overline \Pi$ with respect to $(Z,\mu)$}\label{divis}

For  $\Pi\in {\cal C}_G(\omega)$ the quotient space
$\overline \Pi$ is a $Gl_a(2)$-module by definition. Let $Z$ denote the center of $Gl(2)$ 
considered as our fixed subgroup $Gl(2)\subseteq Gl_a(2)$. Let $B$ denote the subgroup of upper triangular matrices in $Gl(2)$.

\begin{lem}\label{NOINV}
Suppose $\Pi \in {\cal C}_G(\omega)$ admits a nontrivial split Bessel model and is  irreducible 
and not one of the generic representations of type \nosf{VII,VIIIa, IXa}.
Then for any character $\mu$ of $Z$ the subspace $M$  of $\overline\Pi$
spanned by $\mu$-eigenvectors  of $Z\subseteq Gl(2)$ in the representation space of $\overline \Pi \in {\cal C}_{Gl(2)}$  is trivial:
$$ M=\overline\Pi^{(Z,\mu)} = 0 \ .$$
\end{lem}

\begin{proof} Suppose $M\neq 0$.
The center $Z$ of $Gl(2)$ consists in $G$ of the elements $$z_\lambda=[\begin{smallmatrix} \lambda & 0 \cr 0 & \lambda
\end{smallmatrix}] = x_\lambda \cdot \tilde t$$ for $\tilde t = diag(\lambda, 1, 1, \lambda)\in \widetilde T$
and $x_\lambda\in T$.
Since $\tilde T$ commutes with $S$, we have $z_\lambda s_b z_{\lambda}^{-1} = s_{\lambda b}$.
Hence $M\subseteq \overline\Pi^S$, by the argument of the proof of lemma \ref{Restrict-extensions}. 
Since $M$ is a $Gl(2)$-submodule of the $Gl_a(2)$-module $\overline\Pi$ and since the affine subgroup $V\subseteq Gl_a(2)$ is generated by $S \subseteq V$ 
under conjugation by $Gl(2)$, it easily follows that  $M$ is contained in the  subspace of $V$-invariant
elements in $\overline\Pi$. Thus $M\subseteq \overline\Pi$ is contained in the subcategory $i_*({\cal C}_{Gl(2)})$ of ${\cal C}_2$ in the sense of \cite{RW}, 4.0. Recall that $i_*,i^*, j_!,j^!$ are exact functors such that $i^*j_!=0$, $j^!i_*=0$ and $0\to j_!j^! \to id \to i_*i^*\to 0$ holds.
Since $M$ is a submodule of $\overline \Pi$, $i^*(M)$ is a submodule of $i^*(\overline\Pi)$.
Since $i^*(\overline\Pi)\in {\cal C}_{Gl(2)}$ is of finite length (lemma \ref{Atable}), so is $M=i_*i^*(M)$. 
Hence $\dim(k^\rho(M)) < \infty$ for all $\rho$.
Since $k^\rho$ is a left exact functor and $\beta^\rho=k^\rho\circ\eta$ holds (\cite{RW}, lemma 4.16), we obtain inclusions 
$$    k^\rho(M) \hookrightarrow k^{\rho}(\overline \Pi) =\beta^\rho(\Pi) \ .$$
We supposed $M\neq 0$. Hence $i^*(M)\neq 0$ holds  by our previous arguments
and $M=i_*i^*M $ is a constituent of $i_*(B)$ defined as in lemma \ref{Atable}. 
By our assumptions on $\Pi$ we excluded type \nosf{VII,VIIIa, IXa}. So none of the irreducible constituents of
$i^*(M)$ is supercuspidal; see \cite{Roberts-Schmidt}  table A.5, \cite{Roberts-Schmidt_Bessel}, thm. 6.2.2. Hence there exists $\rho$ such that $k^\rho(M)\neq 0$ and hence
$\beta^\rho(\Pi)\neq 0$. This assertion immediately follows from the description of $k^\rho(M)$ in \cite{RW}, lemma 4.3.
By \cite{RW}, prop. 6.3,  $\beta^\rho(\Pi)\neq 0$ is only possible for non-generic $\Pi$ such that 
either $\Pi$ is of type \nosf{IVd, Vd, VIb} and therefore has no split Bessel model (which is excluded by our assumptions); or $\rho$ provides a split Bessel model for $\Pi$.
In this case $\beta^\rho(\Pi)$ is a perfect $TS$-module by \cite{RW}, prop. 6.3.2.
But $\beta^\rho(\Pi)$ is not perfect by \cite{RW}, lemma 3.13, since $\beta^\rho(\Pi)$ contains the finite dimensional submodule $ k^\rho(M)\neq 0$ by our choice of $\rho$ and hence $\kappa(\beta^\rho(\Pi))\neq 0$ holds by the definition of $\kappa$ in \cite{RW}. This is a contradiction, which 
proves $M=0$.
\end{proof} 

\textbf{$(Z,\mu)$-specialization}. $Z({\frak o})$ is compact. Hence for all $M\in {\cal C}_{Gl(2)}$ the natural map $M^{Z({\frak o})} \to M_{Z({\frak o})}$ from the space of $Z({\frak o})$-invariants to the space of $Z({\frak o})$-coinvariants 
is an isomorphism. For unramified characters $\mu$ of $Z$ put
$M^{\mu^{(n)}} = (M^{Z({\frak o})})^{(z_\pi -\mu(\pi))^n}$ and $M_{\mu^{(n)}} = (M_{Z({\frak o})})/{(z_\pi -\mu(\pi))^n}$ where $z_\lambda= \lambda \cdot (\begin{smallmatrix}
1 & 0 \cr 0 & 1\end{smallmatrix})$. We also write $M^{(Z,\mu)}=M^\mu$ resp. 
$M_{(Z,\mu)}=M_\mu$.
Let $\Lambda= \C[\tau, \tau^{-1}]$ denote the ring of Laurent polynomials
in the variable $\tau$. Let $\tau$ act on $M^{Z({\frak o})}$ by
the central element $z_{\pi}^{-1} = \pi^{-1} \cdot (\begin{smallmatrix} 1 & 0 \cr 0 & 1 \end{smallmatrix})$ of $Gl(2)$. This makes the $Gl(2)$-module $M^{Z({\frak o})}$ into a $\Lambda$-module so that both actions commute. 

\medskip
\textbf{Remark}. Since $\Lambda$ is a principal ideal domain,
whose prime elements are the $z_\pi - \mu(\pi)$, the assertion
$\overline\Pi^{(Z,\mu)}=0$ for all smooth characters $\mu$ of $Z$
(as in lemma \ref{NOINV}) implies that $M^{Z({\frak o})}$ is a torsion-free $\Lambda$-module. Since a torsionfree module over a Dedekind ring is flat, 
 $M^{Z({\frak o})}$ then is a flat $\Lambda$-module. 

\medskip
Suppose the assertion
of lemma \ref{NOINV} holds for $\Pi$.
Then for all integers $m$, such that $1\leq m\leq n-1$, the exact sequences 
$$  0 \to \Lambda/( z_\pi-\mu(\pi))^m \to \Lambda/( z_\pi-\mu(\pi))^n \to \Lambda/( z_\pi -\mu(\pi))^{n-m} \to 0 $$
and $\Lambda$-flatness, as explained in the last remark, give by tensoring for all $n$ 
the following exact  sequences showing the \lq{\textit{divisibility property}\rq
 $$   0 \to \overline\Pi_{\mu^{(m)}} \to \overline\Pi_{\mu^{(n)}}  \to \overline\Pi_{\mu^{(n-m)}} 
\to 0 \  .$$   
In particular, the assertion
$\overline\Pi^{(Z,\mu)}=0$  implies  that the kernels of the
natural quotient morphisms $\overline\Pi_{\mu^{(n)}} \twoheadrightarrow \overline\Pi_{\mu^{(n-1)}}$
are isomorphic to $\overline\Pi_{\mu}$ as a $Gl(2)$-module
$$  \overline\Pi_{\mu} \cong Kern(\overline\Pi_{\mu^{(n)}} \twoheadrightarrow \overline\Pi_{\mu^{(n-1)}}) \ .$$ 
\noindent
For irreducible representations $\Pi$ of extended local Saito-Kurokawa type recall from lemma \ref{Atable} the exact sequence
$   0 \to j_!(A) \to \overline\Pi \to i_*(B) \to 0 $
for $A=i_*(\nu)$. Since $\overline\Pi^{(Z,\mu)}=0$ holds for 
 extended local Saito-Kurokawa representations by lemma \ref{NOINV}, this implies
$\overline\Pi^{\mu^{(n)}}=0$ for all $n$. Thus we obtain the following exact sequences 
$$ (*) \quad 0 \to i_*(B)^{\mu^{(n)}} \to j_!(A)_{\mu^{(n)}} \to \overline\Pi_{\mu^{(n)}}
 \to i_*(B)_{\mu^{(n)}} \to 0 \ .$$
 
 \medskip
 \textbf{Notations}. For smooth representations $\alpha$, $\beta$
of $k^*$ (not necessarily characters) we consider the
smooth representation $\alpha\boxtimes\beta$ of $B$ on the tensor product of the representation spaces of $\alpha$ and $\beta$ defined by $(\alpha\boxtimes \beta)(b) = \alpha(t_1) \otimes \beta(t_2)$ for $b=(\begin{smallmatrix} t_1 & * \cr 0 & t_2 \end{smallmatrix})\in B$. The (unnormalized) induced 
representation $Ind_B^{Gl(2)}(\alpha\boxtimes \beta)$ defines a smooth representation of $Gl(2)$
denoted $(\nu^{-1/2}\alpha) \times (\nu^{1/2}\beta)$.  

\medskip 
By definition the functor $j_!$ (see \cite{RW})  is defined by compact induction
from $Gl_a(1)\hookrightarrow Gl_a(2)$. The intersection of the image of $Gl_a(1)$
with $Gl(2)\subseteq Gl_a(2)$ is the mirabolic subgroup of the Borel subgroup
$B\subseteq Gl(2)$ of upper triangular matrices. 
Using double induction, first from the mirabolic subgroup of $B$
to $B$, we obtain the $Gl(2)$-representation $  Ind_B^{Gl(2)}(\nu \boxtimes {\cal S})$.
Indeed $ind_{\{1\}}^{\, k^*}(1) = {\cal S}$  is the Schwartz space 
${\cal S} = C_c^\infty(k^*)$, considered as $k^*$-module with respect to the action $(\lambda \cdot f)(t)=f(\lambda t)$ for $\lambda\in k^*$. In our situation, for representations $\Pi$
of extended Saito-Kurokawa type, we have $A=i_*(\nu)$ in the sense of lemma \ref{Atable}.
For such $A$, more generally for any $A=i_*(\chi)$ given
by some smooth representation $\chi$ of $Gl(1)$, we have 

\begin{lem} \label{Schw} For $A=i_*(\chi)\in {\cal C}_1$  the restriction $j_!(A)\vert_{Gl(2)} $ of $j_!(A)\in {\cal C}_2 $ to $Gl(2)$ can be identified
with the  unnormalized induced representation from the Borel subgroup $B$ of upper triangular matrices 
$  Ind_B^{Gl(2)}(\chi \boxtimes {\cal S})$. 
\end{lem}

Let us return to our case $A=i_*(\nu)$.
Notice ${\cal S} \cong \nu^{-1}\otimes{\cal S}$ as $k^*$-module. Hence in ${\cal C}_{Gl(2)}$
$$ Ind_B^{Gl(2)}(\nu \boxtimes {\cal S})\cong   Ind_B^{Gl(2)}(\nu \boxtimes (\nu^{-1} \otimes{\cal S})) \ .$$
 The space on the right 
consists of smooth functions $f: Gl(2) \to {\cal S}$ such that $f(gb)= \nu(a)\nu^{-1}(d){\cal S}(d) f(g)$ holds
for $b\in B$ with diagonal entries $a$ and $d$, where functions in ${\cal S}$ will be written as functions 
of the variable $x\in k^*$.
Then $${\cal S}^{{\frak o}^*} \ =\ \bigoplus_{n\in \Z} \ \C\cdot 1_{\pi^n \cdot {\frak o}^*}\ $$
for the characteristic functions $1_{\pi^n \cdot {\frak o}^*}$ of the subsets $\pi^n \cdot {\frak o}^*$ of $k^*$.
We view functions  of the induced representation $Ind_B^{Gl(2)}(\nu \boxtimes (\nu^{-1} \otimes{\cal S}))$ as $\C$-valued functions $f(g,x)$. In this sense 
$z_\lambda\in Z\subseteq Gl(2)$ acts on $f(g,x)$ by $(z_\lambda f)(g,x)= f(gz_\lambda,x)= f(z_\lambda g,x)=
\nu(\lambda) \nu^{-1}(\lambda)f(g, \lambda x)= f(g,\lambda x)$. Thus functions $f(g)\in
 Ind_B^{Gl(2)}(\nu \boxtimes \nu^{-1} \otimes{\cal S})^{Z({\frak o})}$ have values in $\Lambda:= \bigoplus_{n\in \Z} \C\cdot 1_{\pi^n \cdot {\frak o}^*}$. The action of $Z$ on $f:G \to \Lambda$ factorizes over $Z/Z({\frak o}) \cong \tau^\Z$ for $\tau= z_\pi^{-1}$. Then $\tau^i(\sum a_j t^j) = \sum_j a_jt^{i+j}$,
and so we may identify $\Lambda$ with the group ring $\C[\tau, \tau^{-1}]$ of $\tau^\Z$, 
or we may think of $\Lambda$ as the ring of Laurent polynomials $\C[t,t^{-1}]$
where $t^n$ corresponds to the function $1_{\pi^n \cdot {\frak o}^*}$ by viewing 
$t=\tau$ as the indeterminate. In this sense 
$ Ind_B^{Gl(2)}(\nu \boxtimes \nu^{-1} \otimes{\cal S})^{Z({\frak o})}$ becomes the $\Lambda$-module
$$   I_B^{Gl(2)}(\sigma_\Lambda) := Ind_B^{Gl(2)}\Bigl(\nu \boxtimes (\nu^{-1}\otimes \Lambda)\Bigr)  \ $$
of smooth $\Lambda$ valued function $f:G\to \Lambda$ such that $f(bg)= \sigma_\Lambda(b)f(g)$
holds for the character $$\sigma_\Lambda: B \to \Lambda^*$$ that sends $b\in B$ with diagonal entries $a$ and $d$ 
to $\nu(a/d)\cdot t^{-v(d)} \in \Lambda^*$.
The action of $G$ commutes with the action of $\Lambda$.
For any $\Lambda$-module $L$ this defines a smooth representation
in ${\cal C}_{Gl(2)}$ via the tensor product
 $$ I_B^{Gl(2)}(\sigma_\Lambda) \otimes_\Lambda L \ .$$
 Since the functions in  $I_B^{Gl(2)}(\sigma_\Lambda)$ are smooth $\Lambda$-valued
 with compact support, the tensor product with $L$ commutes with induction. Hence
 the elements of $I_B^{Gl(2)}(\sigma_\Lambda) \otimes_\Lambda L $ can be considered
 as $L$-valued smooth functions $f$ on $G$ with the property $$f(bg)=\sigma_\Lambda(b)f(g)\ .$$ 

\bigskip\noindent
For any fixed  unramified character $\mu: Z \to \C^*$ 
let $\mathsf{T}$ denote the prime element  $\mathsf{T} = z_\pi - \mu(\pi)$
of $\Lambda$. For any $\Lambda$-module we consider the 
decreasing $\mathsf{T}$-filtration $F_i(L)= \mathsf{T}^i \cdot L$ for $i=0,1,...$ 
with $F_0(L)=L$. For $L=\Lambda$ 
we have $F_0(\Lambda)/F_n(\Lambda) 
=\Lambda/\mathsf{T}^n\Lambda
=\mu^{(n)}$  as a $\Lambda$-module, i.e. as an unramified $Z$-module. 
With this notation $\overline\Pi_{\mu^{(n)}} = F_0(\overline\Pi^{Z{\frak o}})/
F_n(\overline\Pi^{Z{\frak o}})$.
Unraveling the definitions, for the $\Lambda$-module $L = \mu^{(n)}= \Lambda/ \mathsf{T}^n \Lambda$ we obtain from lemma \ref{Schw}
$$ (j_!(A)\vert_{Gl(2)})_{\mu^{(n)}} \ \cong \ Ind_B^{Gl(2)}(\sigma_\Lambda) \otimes_\Lambda L 
\cong Ind_B^{Gl(2)}(\sigma_\Lambda \otimes_\Lambda L) \ $$
because induction is an exact functor. 
Ignoring the $\Lambda$-structure, as a $B$-module $\sigma_\Lambda \otimes_\Lambda L = \nu\boxtimes (\nu^{-1}\mu \otimes 1^{(n)})$
 will be identified with the $\C$-vectorspace $\Lambda/\mathsf{T}^n\Lambda$ 
 such that $b\in B$ acts by $b\cdot v=  \nu(a/d)\mu(d) \tau^{-v(d)} v$
for $v\in \Lambda/\mathsf{T}^n\Lambda $, whereas the center $Z$ acts by $z_\lambda v = \mu(\lambda) \tau^{v(\lambda)} \cdot v$.
Hence, as a representation of $Gl(2)$
$$  (j_!(A)\vert_{Gl(2)})_{\mu^{(n)}} \ \cong \ \nu^{1/2}\times (\nu^{-1/2}\mu\otimes 1^{(n)}) \ .$$
The same arguments show 

\begin{lem}\label{l6.3} For $A=(\nu^s)^{(m)}$ the module $(j_!i_*(A))_{Z,\mu}$ has a
composition series of modules of length $m$ whose graded pieces are of the form
$$(j_!i_*(\nu^s))_{Z,\mu} \cong  \nu^{s-1/2} \times (\nu^{-s+1/2}\mu)\ .$$
Furthermore $(j_!i_*(A))^{Z,\mu}=0$ holds (divisibility property).
\end{lem}

\begin{proof} Induction is an exact functor and $(\nu^s)^{(m)}$
has a composition series of
length $m$  whose graded pieces are $\nu^s$.
Hence $(j_!i_*(A))_{Z,\mu}$ has a composition series of
length $m$ whose graded pieces are $(j_!i_*(\nu^s))_{Z,\mu}$.
Since $(j_!i_*(\nu^s))^{Z,\mu}=0$ by lemma \ref{NOINV},
the claim reduces to the case $m=1$. Then the proof is 
the same as in the special case $s=1$, $n=1$ 
discussed preceding the lemma. 
\end{proof}

\textbf{The $Gl(2)$-modules $\widehat \Pi$}. For $\mu\!=\!\nu^2$ put $\widehat \Pi \!=\! \overline\Pi_{(Z,\mu)}$. So $\widehat \Pi$ is the $Z$-specialization of $\overline \Pi$ with respect to the character $\mu=\nu^2$ of $Z$. 
From $(*)$ in the case $n=1$ we  obtain the following exact sequence  of $Gl(2)$-modules
$$ (**) \quad 0 \to i_*(B)^{(Z,\mu)}  \to \nu^{1/2}\times \nu^{3/2} \to \widehat\Pi
 \to i_*(B)_{(Z,\mu)}  \to 0 \ .$$
Since  $\mu=\nu^2$,
lemma \ref{Atable} implies $i_*(B)_{(Z,\mu)} \cong i_*(B)^{(Z,\mu)}=0 $
for \nosf{IIb, Vbc, XIb} and $i_*(B)_{(Z,\mu)} \cong i_*(B)^{(Z,\mu)}  =\nu\circ\det$ in the cases \nosf{VIc, VId}.  By $(**)$ this immediately gives the next lemma for the first 
three cases.

\begin{lem} \label{needed} The
$Gl(2)$-module $\widehat \Pi$ is isomorphic to $\nu^{1/2}\times \nu^{3/2}$ for the cases \nosf{IIb, Vbc, XIb} and $\widehat \Pi$ is isomorphic to $\nu^{3/2}\times \nu^{1/2}$ for the cases \nosf{VIc, VId}.
\end{lem}

\begin{proof} 
First notice:
$\nu^{1/2}\times \nu^{3/2}$ is an indecomposable with $\nu\circ\det$ as a submodule and $Sp(\nu)$ as quotient:
$  0 \to \nu\circ\det \to \nu^{1/2}\times \nu^{3/2} \to Sp(\nu) \to 0 $.
Hence in the cases \nosf{VIc, VId} the above exact sequence $(**)$ for $\widehat\Pi$ and $i_*(B)^{(Z,\mu)}  =\nu\circ\det$ gives an exact sequence
$$ 0 \to Sp(\nu) \to \widehat\Pi \to \nu\circ\det \to 0 \ .$$
So, for the proof of the lemma it suffices to show $\widehat\Pi\, \not\cong\, Sp(\nu)\, \oplus\, (\nu\circ\det)$.
Consider the commutative diagrams obtained from $(*)$ for $\mu=\nu^2$ and $n=1$ and $n=2$
$$ \xymatrix{    0 \ar[r] & \nu\circ\det \ar[r] &    (j_!(A)\vert_{Gl(2)})_{\mu}       \ar[r]^-{k\ \ } &  \overline\Pi_{\mu}    \ar[r]^-{\lambda} &   \nu\circ\det     \ar[r] & 0  \cr  
 0 \ar[r] & \nu\circ\det \ar@{.>}[u]_0 \ar@{_{(}.>}[dr] \ar[r]^-{u\ } &    (j_!(A)\vert_{Gl(2)})_{\mu^{(2)}}      \ar@{->>}[u]^{pr}  \ar[r]^-{v\ \ } &  \overline\Pi_{\mu^{(2)}}  \ar@{->>}[u]^{pr}   \ar@{->>}[r] &   \nu\circ\det    \ar@{.>}[u]_\sim \ar[r] & 0 \cr 
 & &    (j_!(A)\vert_{Gl(2)})_{\mu}   \ar@{^{(}->}[u]^i    \ar[r] &  \Omega \ar@{^{(}->}[u]^w \ar@{.>}[ur]_0  &    &   } $$ 
where $\Omega$ denotes the kernel of the natural surjective quotient map $\overline\Pi_{\mu^{(2)}}
\to \overline\Pi_{\mu}$. Since lemma \ref{NOINV} holds in the cases VIc, VId, we have the divisibility property
$$ \Omega \cong \overline\Pi_\mu \ .$$
Similarly, the divisibility property following from lemma \ref{l6.3} implies that both vertical
sequences in the middle define short exact sequences.  
The right vertical map is surjective by right exactness of coinvariants. Hence it is bijective by dimension reasons. The left vertical map is zero since ${\mathsf T}=z_\lambda - \mu(\pi)$ acts trivially $\nu\circ \det$ for $\mu=\nu^2$. Hence the image of $u$
is annihilated by ${\mathsf T}$, 
so  $u(\nu\circ\det) $ in the middle horizontal line is contained in ${\mathsf T}\cdot (j_!(A)\vert_{Gl(2)})_{\mu^{(2)}}= i(j_!(A)\vert_{Gl(2)})_{\mu})$  and therefore is annihilated by $pr$. Hence $\nu\circ\det $ defines a submodule
in  $Kern(pr)\cong j_!(A)\vert_{Gl(2)})_{\mu} $, the image of the first lower vertical map $i$. Since induction defines an exact functor, or alternatively by divisibilty, this first lower vertical map is the inclusion obtained from the $\mathsf{T}$-filtration on $1^{(2)}$. 
The same holds for $w$ by the divisibility property mentioned above.
Since $\widehat\Pi =  \overline\Pi_\mu$ has lenght two with the constituents $\nu\circ\det$ and $Sp(\nu)$, 
$\overline\Pi_{\mu^{(2)}}$ has length 4 as a $Gl(2)$-module by divisibility. From this
it is not hard to unravel that $\Omega$ can be expressed purely in terms
of the module $I:= (j_!(A)\vert_{Gl(2)})_{\mu^{(2)}}$, meaning
$$  \Omega \ \cong\   \frac{Kern\bigl(k\circ pr: (j_!(A)\vert_{Gl(2)})_{\mu^{(2)}} \twoheadrightarrow Kern(\lambda)\bigr)}{ u(\nu\circ\det) } \ .$$  
$I$ has also length 4 as a $Gl(2)$-module with two constituents $Sp(\nu)$ and two
constituents $\nu\circ\det$. The quotient $E:= I/u(\nu\circ\det)$ has three constituents.
$Kern(\lambda) \cong Sp(\nu)$ is a natural quotient of $E$ (under the map
$\overline b$ of the next diagram) and the kernel $J:= Kern(\overline b: E \to Sp(\nu))$
in the next diagram by its definition is isomorphic to $\Omega$. Using this, our claim
$\widehat\Pi\cong \nu^{3/2}\times\nu^{1/2}$ follows from
$\widehat\Pi= \overline\Pi_\mu \cong \Omega \cong J$ and the next lemma \ref{J}.

We remark that our arguments more generally show: If a  
smooth representation $M$ of $Gl(2)$ fits into an exact sequence 
$0 \to j_!i_*(\nu)|_{Gl(2)} \to M \to (\nu\circ\det)\to 0$ of $Gl(2)$-modules
and $M^{(Z,\nu^2)}=0$ holds, then its central specialization $M_{(Z,\nu^2)}$
is isomorphic to $\nu^{3/2}\times\nu^{1/2}$. 
 \end{proof} 

The $Gl(2)$-module $I=(j_!(A)\vert_{Gl(2)})_{\mu^{(2)}}$ is 
isomorphic to $\nu^{1/2} \times (\nu^{3/2} \otimes 1^{(2)})$; see the remarks
preceding lemma \ref{l6.3} for $n=2$. This isomorphism and the exact sequence
$0\to 1 \to 1^{(2)}\to 1 \to $ defined by the ${\mathsf T}$-filtration 
induces the next exact sequence $(***)$ 
with submodule $F_1/F_2\cong \nu^{1/2} \times \nu^{3/2}$ and quotient $F_0/F_1\cong \nu^{1/2} \times \nu^{3/2}$
(using the $\mathsf{T}$-filtration on $1^{(2)}$ and the exactness of induction)
$$  (***) \quad 0 \to \nu^{1/2} \times \nu^{3/2} \to I
\to \nu^{1/2} \times \nu^{3/2}  \to  0 \ .$$ 
There are morphisms  $a: \nu\circ\det \hookrightarrow F_1/F_2 \hookrightarrow I $ and $b: I \twoheadrightarrow F_0/F_1 \twoheadrightarrow 
Sp(\nu)$, unique up to nonvanishing constants. 
For the cokernel $E$ of $a$, the morphisms $F_1/F_2 \hookrightarrow I $ and $b$ induce the maps $\overline a$ and $\overline b$ of the commutative diagram 
 \label{IEJ}
$$ \xymatrix{ &  &  Sp(\nu)  \ar@{=}[r] & Sp(\nu)  & \cr
0 \ar[r] &  Sp(\nu) \ar[ur]^0\ar[r]^{\overline a} &  E  \ar@{->>}[r]\ar@{->>}[u]_{\overline b} &  \nu^{1/2} \times \nu^{3/2}\ar@{->>}[u]_c \ar[r] & 0\cr
0 \ar[r] &  Sp(\nu) \ar@{=}[u]\ar[r] &  J  \ar[r]\ar@{^{(}->}[u] & \nu\circ\det \ar[r]\ar@{^{(}->}[u] & 0}
 \ $$
The map $c$ is defined as the composition $ \nu^{1/2} \times \nu^{3/2} \cong F_0/F_1 \twoheadrightarrow Sp(\nu)$.
Now, for the kernel $J$  of the morphism $\overline b$
the following holds 

\goodbreak 
\begin{lem} \label{J} \ $J\cong \nu^{3/2} \times \nu^{1/2}$.
\end{lem}

\begin{proof} $E$ has three constituents, namely $\nu\!\circ\!\det$ and twice $Sp(\nu)$.
Hence $J$ has the two constituents $\nu\!\circ\!\det$ and $Sp(\nu)$.  
By Frobenius reciprocity either $J\cong \nu\!\circ\!\det\oplus Sp(\nu)$ splits, 
or $J$ is indecomposable so that $J\cong  \nu^{3/2}\times\nu^{1/2}$ holds. 

\bigskip\noindent
Suppose $J$ splits. Then the induced representation $I$ with ${\mathsf T}$-filtration $(***)$ must contain a submodule 
$M$ with two constituents of the form $\nu\!\circ\!\det$, so that these constituents occur in different
layers of the filtration $(***)$ induced by the exact sequence coming from the ${\mathsf T}$-filtration $F_\bullet$. Hence the center $Z$ acts nontrivially on
$(\nu^{-1}\circ\det)\otimes M$. This implies $M \cong (\nu\!\circ\!\det) \otimes 1^{(2)}$ as a $Gl(2)$-module, where $Gl(2)$
acts on $1^{(2)}$ by unipotent matrices involving the logarithm of $\nu\circ\det$. From the splitting of $J$ we hence conclude that
$$  Hom_{Gl(2)}( M , I) \supseteq Hom_{Gl(2)}( M, M)  \cong Hom_{Gl(2)}( 1^{(2)}, 1^{(2)}) \ .$$  
has dimension at least 2. Indeed, the right side has dimension 2. 
%
On the other hand $I\cong \nu^{1/2} \times (\nu^{3/2} \otimes 1^{(2)})$. Hence  by Frobenius
reciprocity
$$ Hom_{Gl(2)}(M,I) = Hom_{Gl(2)}\bigl( M , Ind_B^{Gl(2)}(\nu \boxtimes (\nu \otimes 1^{(2)}) )\bigr) \ $$
 is isomorphic to the vectorspace
$$ Hom_B\bigl(M, \nu \boxtimes (\nu \otimes 1^{(2)})\bigr)  \ .$$
The representation spaces of $M$ and $\nu \boxtimes (\nu \otimes 1^{(2)})$ are both isomorphic to $\C^2$ as vectorspaces.
So, homomorphisms in  $ Hom_B(M, \nu \boxtimes (\nu \otimes 1^{(2)}))$ are given
by $2\times 2$-matrices that satisfy for all $t_1,t_2\in k^*$ the conditions
$$ \begin{bmatrix} \alpha & \beta \cr \gamma & \delta\end{bmatrix} \nu(t_1t_2) \begin{bmatrix} 1 & \log \vert t_1t_2 \vert \cr 0 & 1\end{bmatrix} \ = \ \nu(t_1)\nu(t_2)\begin{bmatrix} 1 & \log \vert t_2\vert \cr 0 & 1\end{bmatrix}  \begin{bmatrix} \alpha & \beta \cr \gamma & \delta\end{bmatrix}  \ .$$
Notice, these equations are equivalent to the assertion that the matrix in brackets
with the entries $\alpha, \beta, \gamma, \delta\in \C$ defines a $B$-linear map from $M\vert_B$ to $ \nu \boxtimes (\nu \otimes 1^{(2)})$
for the action of $(\begin{smallmatrix} t_1 & * \cr 0 & t_2 \end{smallmatrix}) \in B$
defined by the $B$-modules $M\vert_B$ resp. $ \nu \boxtimes (\nu \otimes 1^{(2)})$.
Solving these equations  for  $\alpha, \beta, \gamma, \delta$ gives 
$\alpha=\delta=\gamma=0$. Hence the dimension of
$Hom_B(M, \nu \boxtimes (\nu \otimes 1^{(2)}))$ is at most one, as well as the dimension of 
$Hom_{Gl(2)}(M, I)$ as shown by Frobenius reciprocity. This contradicts 
the dimension estimate $\geq 2$ found earlier. This contradiction implies
that $J$ cannot split. Hence $J\ \cong\  \nu^{3/2}\times\nu^{1/2}$
follows.
\end{proof}


\begin{prop} \label{ima}
The image of the paramodular new vector $v_{new}$ of $\Pi$ in $\widehat\Pi$ under the composite projection $\Pi\twoheadrightarrow \overline\Pi \twoheadrightarrow \widehat \Pi$  spans the one-dimensional space $\widehat\Pi^K$ of spherical vectors in  $\widehat\Pi$.
\end{prop}
 
\begin{proof} By lemma \ref{needed},  for $K=Gl(2,{\frak o})$ the subspace of $K$-spherical vectors in $\widehat \Pi$
has dimension one for normalized representations of extended Saito-Kurokawa type.
Consider the functionals $f:\Pi \to (\C,\nu\boxtimes\nu)$ of lemma \ref{new=sur}. By their equivariance
properties they factorize over the two quotient map $\eta: \Pi \twoheadrightarrow \overline\Pi$ and also
over the $(Z,\nu^2)$-specialization  $\overline\Pi \twoheadrightarrow \widehat \Pi$. Indeed $Z$ is contained in $T_{split} = \tilde T \cdot T$
 and the character $\nu\boxtimes\nu$ of $\tilde T\times T$ restricts to $\mu=\nu^2$ on $Z$.
By lemma \ref{new=sur} the image $f(v_{new})$ is nonzero, hence the image $\widehat v_{new}$ of $v_{new}$ in $\widehat \Pi$ is also nonzero. Since the images of $K_H$-spherical vectors of $\Pi$ are $K$-spherical in $\widehat\Pi\in {\cal C}_{Gl(2)}$, this implies our assertion. \end{proof}

\section{The comparison isomorphism}

For a smooth character $\chi$ of $k^*$ and a smooth representation $\pi$
of $Gl(2)$ consider the quotient map defined by the $(T,\chi)$-coinvariants $$\pi\twoheadrightarrow \pi_{(T,\chi)}$$
from the representation space $\pi$ to the maximal quotient space $ \pi_{(T,\chi)}$
on which $T \subseteq Gl(2)$ acts by the character $\chi$. We identify $T$ with 
$k^*$ by $T\ni x_\lambda \mapsto \lambda\in k^*$
for  $x_\lambda = [\begin{smallmatrix}  \lambda & 0 \cr 0 & 1
\end{smallmatrix}]$ in $Gl(2)$.  
$T$ is conjugate in $Gl(2)$ to the subtorus of $Gl(2)$ used 
in lemma \ref{Waldspurger-Tunnell}, so the results from that lemma show
$\dim(\pi_{(T,\chi)})=1$ for all irreducible generic representations $\pi\in {\cal C}_{Gl(2)}$.

\bigskip\noindent
For $\pi=\widehat \Pi$ and extended Saito-Kurokawa 
representations $\Pi$ and $\chi=\nu^2$ 
$$  \widehat\Pi_{(T,\nu^2)} \cong \C \ $$ 
holds since
$\dim(\pi_{(T,\chi)})=0$  for $\pi=\nu\!\circ\!\det$,
and hence $\widehat\Pi_{(T,\nu^2)} \cong  Sp(\nu)_{(T,\nu^2)} \cong \C$
by lemma \ref{needed}. Recall $\widehat T = \overline\Pi_{(Z,\mu)}$.
If $\mu=\chi\rho$, then $\widehat\Pi_{(T,\nu^2)}$ is the maximal quotient space
of $\overline\Pi$ on which the diagonal split torus in $Gl(2)$ acts with the character
$\chi\boxtimes\rho$. Notice, this diagonal torus is generated by $T$ and $Z$.

\bigskip\noindent
\textbf{Claim}. \textit{ For  normalized $\Pi$ of generalized Saito-Kurokawa type and $\rho=1$, $\chi=\nu^2$ there exists a unique $(T,\nu^2)$-equivariant map
$WT: \widehat \Pi \to k_{\nu^2}(k_1(\overline\Pi))$ 
making the following diagram commutative }
$$ \xymatrix{ & &  \overline\Pi \in {\cal C}_{Gl_a(2)} \ar@{->>}[dl]_{k_1} \ar@{->>}[dr]^{(Z,\nu^2)}    &  &   \cr     &
{\cal C}_{Gl_a(1)} \ni \widetilde \Pi \ \ \ \ \ \ar@{->>}[dr]_{k_{\nu^2}} &           &   \ \ \  \ \ \widehat \Pi \in {\cal C}_{Gl(2)} \ar@{->>}[dl]^{\ \ \ WT} & \cr
&   &     \widetilde \Pi_{(T,\nu^2)} &  & \cr} $$

\begin{proof}
By definition, $ \widetilde \Pi_{(T,\nu^2)}  = k_\chi(k_\rho(\overline\Pi))$  is the maximal quotient space
of $\overline\Pi$ on which the Borel subgroup $B$ of $Gl(2)$ acts by the character
$\chi\boxtimes\rho = \nu^2\boxtimes 1$. On this quotient the diagonal torus generated by $T$ acts by $\nu^2$,  and $Z$ acts by $\mu=\nu^2$ and  $\tilde T$ acts trivially. Since $Z$ acts on $\widehat \Pi$ by the central character $\mu=\nu^2$, the universal property of coinvariant quotients for $\rho=1$  shows that there exists a \textit{surjective} $\C$-linear map
$$  \phi:   \widehat\Pi_{(T,\nu^2)} \ \twoheadrightarrow \ k_{\nu^2}(k_\rho(\overline\Pi)) = k_{\nu^2}(\beta_\rho(\Pi)) \ .$$
All normalized extended Saito-Kurokawa representation have a split Bessel model for the Bessel character
$\rho=1$. Furthermore $\beta_\rho(\Pi)$ is a perfect $Gl_a(1)$-module
(\cite{RW} cor. 6.10). For perfect $Gl_a(1)$-modules $M$, $\dim(M_\chi)=1$ holds for all smooth characters $\chi$.
Therefore
$$   k_{\nu^2}(\beta_1(\Pi)) \cong \C \ .$$
Hence the above comparison map $\phi$ is an isomorphism by dimension reasons.
The natural projection map $\widehat\Pi \twoheadrightarrow   \widehat\Pi_{(T,\nu^2)}$ composed with $\phi$ gives us the desired
map $$WT: \widehat\Pi \to  k_{\nu^2}(k_1(\overline \Pi))$$ 
making the above diagram commutative. \end{proof} Furthermore

\begin{lem} \label{final1}
This map $WT$ 
defines a universal $(T,\nu^2)$-coinvariant quotient map on $\widehat \Pi$. It vanishes
on $\widehat\Pi^K = \C \cdot \widehat v_{new}$ for $\Pi$ of type \nosf{IIb, Vbc, XIb} and it is nontrivial
on $\widehat\Pi^K = \C \cdot \widehat v_{new}$ if $\Pi$ of type \nosf{VIc, VId}.
\end{lem}

\begin{proof}
The first assertion is obvious from the proof of the last claim. The remaining assertions follow from lemma \ref{Waldspurger-Tunnell}, lemma \ref{needed}
and lemma \ref{ima}.
\end{proof}

Up to a nonvanishing scalar, by \cite{RW}, lemma 3.31
the universal quotient map $p_{\nu^2}: \widetilde \Pi \to k_{\nu^2}(\widetilde\Pi) \cong \C$
is given by the normalized zeta integral of some model of the perfect $Gl_a(1)$-module
$\widetilde\Pi$ 
$$  p_{\nu^2}(f) =  \lim_{s\to 0} \frac{ Z(f, \nu^{-2}, s)}{L(\widetilde \Pi,\nu^{-2},s)}  \ .$$
Since $Z(f, \nu^{-2}, s) = Z^{PS}_{reg}(f, \nu^{-1/2}, s)$ and $L(\widetilde \Pi,\nu^{-2}, s)
= L^{PS}_{reg}(\Pi,\nu^{-1/2}, s)$, lemma \ref{final1} together with the claim above (commutativity of the last diagram) implies that 
$ord_{s=-\frac{1}{2}}(Z^{PS}_{reg}(s, W_{new}, 1)$ is $\geq 1$ for the cases \nosf{IIb, Vbc, XIb}
and that $ord_{s=-\frac{1}{2}}(Z^{PS}_{reg}(s, W_{new}, 1)$ is $0$ for \nosf{VIc, VId}.
By corollary \ref{Hilfsschritt} for \nosf{VIc, VId} hence
$$ {\cal L}(s) = L(s+\frac{1}{2})^2\ .$$ 
A modification of these arguments, that will be given in the next section, similarly implies for the cases \nosf{IIb, Vbc, XIb} 
 $ord_{s=-\frac{1}{2}}(Z^{PS}_{reg}(s, W_{new}, 1)= 1$,   
hence 
$$ {\cal L}(s) = L(s+\frac{1}{2})\ .$$
\textbf{A first order deformation argument}. 
For the cases \nosf{IIb, Vbc, XIb} and the characters $\rho=1$, $\mu=\chi=\nu^2$ we now construct a map  $WT'$ making the diagram  
$$ \xymatrix{ & &  \overline\Pi \in {\cal C}_{Gl_a(2)} \ar@{->>}[dl]_{k_1} \ar@{->>}[dr]^{(Z,\mu^{(2)})}    &  &   \cr     &
{\cal C}_{Gl_a(1)} \ni \widetilde \Pi \ \ \ \ \ \ar@{->>}[dr]_{k_{\chi^{(2)}}} &           &   \ \ \  \ \ \widehat \Pi' \in {\cal C}_{Gl(2)} \ar@{->>}[dl]^{\ \ \ WT'} & \cr
&   &     \widetilde \Pi_{(T,\chi^{(2)})} &  & \cr} $$
commutative. Concerning the notations first notice $z_\pi=x_\pi \tilde t$ for $\tilde t\in \tilde T$, where $(t_1,t_2)=(\pi, 1)$. So, if $\rho$ acts trivially, then $x_\pi$ and $z_\pi$ can be identified. 
This being said, we now explain the notations appearing in the diagram above:
\begin{itemize}
\item $\widehat\Pi'$ is the maximal quotient of $\overline\Pi$ on which $T({\frak o})$ and
$\mathsf{T}^2$ for $\mathsf{T}=z_\pi - \mu(\pi)$ acts trivially.
\item $\widetilde \Pi_{(T,\chi^{(2)})}$ is the maximal quotient of $\widetilde \Pi=\beta_1(\Pi)$ on which $T({\frak o})$ and the operator 
$(x_\pi - \chi(\pi))^2$ acts trivially. 
\item $k_{(T,\chi^{(2)})}$ the natural projection map, defined by the
$s$-derivative of the zeta integral on a model of $\widetilde\Pi$ (\cite{RW}, lemma 3.31).
\item $pr': \widehat \Pi' \to \widehat\Pi'_{(T,\chi^{(2)})}$ is the maximal quotient of $\widehat \Pi'$ on which $T({\frak o})$ and the operator
$(x_\pi - \chi(\pi))^2$ acts trivially.
\end{itemize}

To construct $WT'$ we use similar arguments as for the construction of $WT$.
It is easy to show
$  \widehat\Pi'_{(T,\chi^{(2)})} \cong \C^2  $.
Furthermore $ \widetilde \Pi_{(T,\chi^{(2)})} \cong \C^2$
 (\cite{RW}, lemma 3.34 and primitivity). 
Since $\rho=1$, $\tilde T$ acts trivially on $\widetilde\Pi=\beta_1(\overline \Pi)$. Hence
$x_\pi - \chi(\pi)=z_\pi- \mu(\pi)$ on $\tilde\Pi$. Therefore
$(z_\pi- \mu(\pi))^2=0$ holds on $\tilde\Pi_{(T,\chi^{(2)})}$.
So the composition $k_{\chi^{(2)}}\circ k_1$ on the left
canonically factorizes over the universal quotient map $\overline \Pi \to \widehat\Pi'$.
Since $(x_\pi-\chi(\pi))^2$ acts trivially on $ \widetilde \Pi_{(T,\chi^{(2)})}$, this map further factorizes over the quotient $\widehat\Pi' \twoheadrightarrow \widehat\Pi'_{(T,\chi^{(2)})}$, inducing a canonical surjective $\C$-linear
 quotient map $$\phi': \widehat \Pi'_{(T,\chi^{(2)})} \twoheadrightarrow \widetilde \Pi_{(T,\chi^{(2)})} $$
by the universal property of coinvariants. 
Then  $WT' := (\phi')\circ pr'$ makes the above diagram commutative
and is a surjective $\C$-linear map since $k_{\chi^{(2)}}\circ k_1$ is surjective. 
By the same reason $\phi'$ is surjective, so counting dimensions shows 
$$ \phi': \widehat \Pi'_{(T,\chi^{(2)})} \cong \widetilde \Pi_{(T,\chi^{(2)})} \ $$
must be an isomorphism.
We use this  to show

\begin{lem} \label{last} For the cases \nosf{IIb, Vbc, XIb}
the map $WT'$ is nonzero on the space of $K$-spherical
vectors $(\widehat\Pi')^K$,  and $(\widehat\Pi')^K$ contains the image $\widehat v_{new}' $ in $\widehat \Pi'$
of the new vector $v_{new}$ from $\Pi$.
\end{lem}

\begin{proof} Notice, as a representation of $Gl(2)$ there is the exact sequence 
$$  0 \to \widehat\Pi \to \widehat\Pi' \to \widehat\Pi \to 0 \ $$
defined by the $\mathsf{T}$-filtration on $\overline\Pi^{Z({\frak o})}$.
Indeed $\widehat\Pi' = (\overline\Pi^{Z({\frak o})})_{\mu^{(2)}}$
and the quotient on the right side is $\widehat\Pi = (\overline\Pi^{Z({\frak o})})_{Z,\mu}$.
The submodule on the left is $\mathsf{T}\cdot (\overline\Pi^{Z({\frak o})})_{\mu^{(2)}}$.
By the divisibility property, discussed in section \ref{divis} (i.e. by lemma \ref{NOINV}), this $Gl(2)$-module on
the left is isomorphic to  $\widehat\Pi$.   Recall $\widehat\Pi
= \nu^{1/2}\times \nu^{3/2}$ by lemma \ref{needed}, whose proof also
implies $(j_!(A)\vert_{Gl(2)})_{\mu^{(2)}} \cong \widehat\Pi'$ for $A=i_*(\nu)$. 
Hence  $\widehat\Pi'$ as a representation of $Gl(2)$
is isomorphic to $I$, as given in the exact sequence $(***)$ on page \pageref{IEJ}.  
So $\widehat\Pi$ has $\nu\circ\det$ as unique nontrivial irreducible submodule.
The map $WT'$ is trivial on the submodule $\nu\circ\det \hookrightarrow   \widehat\Pi \cong \mathsf{T}\cdot \widehat\Pi' $, hence factorizes over the quotient representation $E=I/a(\nu\circ\det)$. But $E^K =J^K$, hence $\dim(E^K)=1$ by lemma 
\ref{J}. Hence  $E^K$ is generated by the
image of the new vector $v_{new}$ by lemma \ref{ima}. 
Since $WT'$ restricted to $J$ gives the usual Waldspurger-Tunnel
quotient map for $J$, we can invoke lemma \ref{Waldspurger-Tunnell}.
If we utilize $J \cong \nu^{3/2} \times \nu^{1/2} $ (lemma \ref{J}), this implies that $WT'$ is nonzero on 
the new vector.
This proves our claim.  
\end{proof}

Still assume that $\Pi$ is normalized of type \nosf{IIb, Vbc, XIb}. 
Let us now explain the consequence of the last lemma for the exceptional
pole order. Lemma \ref{last} and the commutativity of the last diagram imply that $$k_{\chi^{(2)}}:  \widetilde\Pi \twoheadrightarrow
\widetilde \Pi_{(T,\chi^{(2)})}$$ does
not vanish on the image  of $\widetilde v_{new}\in \widetilde\Pi$. Hence
the description of $k_{\chi^{(2)}}$ in terms of the derivative of the normalized zeta integral on $\widetilde\Pi$
of \cite{RW}, lemma 3.34 implies the following: Although $Z^{PS}_{reg}(s, f_{new}, 1)$  vanishes at $s=-\frac{1}{2}$ on the Bessel function $f_{new}$ of the new vector $v_{new}$ attached to $\tilde v_{new}$, its 
 first derivative $\frac{d}{ds}Z^{PS}_{reg}(s, f_{new},1)$ at $s=-\frac{1}{2}$ does not vanish on 
$f_{new}$. Therefore $$k=ord_{s=-\frac{1}{2}} Z^{PS}_{reg}(s, f_{new}, 1)   = 1$$
holds for the cases \nosf{IIb, Vbc, XIb}. This determines the pole order of $Z^{PS}_{ex}(s,\Pi,\Lambda)$
at $s_0$ for the cases \nosf{IIb, Vbc, XIb} by corollary \ref{Hilfsschritt}. 
Since this order is one, together with lemma \ref{final1} this finally completes the proof of the following main theorem.

\begin{thm}\label{THEEND}
The exceptional $L$-factors $L_{ex}^{PS}(s,\Pi,\Lambda)$ for smooth irreducible representations $\Pi$ of $G$ 
with Bessel model $\Lambda$ are given by table 1. 
\end{thm} 

In table 1 we used the notation of \cite{Sally-Tadic} and \cite{Roberts-Schmidt} 
for the classification of the irreducible representations of $G$.

\begin{proof} We may assume that $\Pi$ is normalized. We use our computation
of the auxiliary $L$-factor ${\cal L}(s)$, which is the regularizing $L$-factor
of $L^{PS}(s,W_v,\Phi_0,\mu)$ for all Bessel functions $W_v$ fixing $\Phi_0$ to be the characteristic
function  of the lattice ${\frak o}^2 \oplus {\frak o}^2$.
As stated at the end of section 7, we have  $ {\cal L}(s) = L(s+\frac{1}{2})$ for the cases \nosf{IIb, Vbc, XIb} and by [RW2] $L_{reg}^{PS}(s,\Pi,\Lambda) = L(\mu\otimes M,s)$. Hence $L_{ex}(s,\Pi,\Lambda) = {\cal L}(s)$. Thus $L_{ex}(s,\Pi,\Lambda) = L(s+\frac{1}{2})$ holds for the cases \nosf{IIb, Vbc, XIb}.
Recall $ {\cal L}(s) = L(s+\frac{1}{2})^2$ for \nosf{VIc, VId}. Since $L_{sreg}^{PS}(s,\Pi,\Lambda) = L(s+\frac{1}{2})$ holds in these cases by [RW2], this implies $$L_{ex}(s,\Pi,\Lambda) = {\cal L}(s)/L_{sreg}^{PS}(s,\Pi,\Lambda) = L(s+\frac{1}{2}) $$ in the two remaining cases \nosf{VIc, VId}. This determines $L_{ex}(s,\Pi,\Lambda)$ in all the cases, where it is nontrivial.
\end{proof}

\textbf{Remark}. Notice $L^{PS}(s,\Pi,\mu,\Lambda)= L^{PS}(s,\mu\otimes\Pi,\mu\otimes\Lambda)$ holds.
Hence the last theorem \ref{THEEND} describes all the exceptional $L$-factors $L^{PS}(s,\Pi,\mu,\Lambda)$
for $\Pi$ , $\mu$ and split Bessel models $\Lambda$ in the sense of 
Piateskii-Shapiro.

\medskip\noindent
\textbf{Remark}. For cuspidal $\Pi$ and split Bessel models $L^{PS}(s,\Pi,\Lambda)=1$ holds for the local $L$-factor. The analogous statement does not hold for anisotropic Bessel models $\Lambda$ as shown in
\cite{Danisman3bis}. 

\goodbreak
\section{Bibliography}

\begin{footnotesize}
\bibliographystyle{amsalpha}

\end{footnotesize}

\bigskip\noindent

\bigskip\noindent

\begin{footnotesize}

\centering{Rainer Weissauer\\ Mathematisches Institut, Universit\"at Heidelberg\\ Im Neuenheimer Feld 205, 69120 Heidelberg\\ email: weissauer@mathi.uni-heidelberg.de}

\end{footnotesize}

\
\end{document}